\documentclass[a4paper]{amsart}
\usepackage{amsmath,amssymb,amsthm,enumerate}
\title[Self-dual Einstein ACH metrics]
{Self-dual Einstein ACH metrics and CR GJMS operators in dimension three}
      
\author{TAIJI MARUGAME}
\date{}

\newcommand\R{\mathbb{R}}
\newcommand\Ric{{\rm Ric}}
\newcommand\scal{{\rm Scal}}
\renewcommand\a{\alpha}
\renewcommand\b{\beta}
\newcommand\g{\gamma}
\renewcommand\d{\delta}
\newcommand\e{\epsilon}
\renewcommand\r{\rho}
\renewcommand\th{\theta}

\newcommand\bh{\boldsymbol{h}}
\newcommand\bth{\boldsymbol{\theta}}
\newcommand\U{\Upsilon}

\newcommand\pa{\partial}
\newcommand\ol{\overline}
\newcommand{\calE}{\mathcal{E}}

\newcommand{\wt}{\widetilde}
\newcommand{\wh}{\widehat}

\newtheorem{lem}{Lemma}[section]
\newtheorem{thm}[lem]{Theorem}

\newtheorem{prop}[lem]{Proposition}

\theoremstyle{definition}
\newtheorem{dfn}[lem]{Definition}
\newtheorem{rem}[lem]{\it Remark}

\numberwithin{equation}{section}

\makeatletter
\@addtoreset{equation}{section}

\makeatother
\address{Institute of Mathematics, Academia Sinica, Astronomy-Mathematics Building, No. 1, Sec. 4, Roosevelt Road, Taipei 10617, TAIWAN }
\email{marugame@gate.sinica.edu.tw}



\begin{document}

\begin{abstract} 
By refining Matsumoto's construction of Einstein ACH metrics, we construct a one parameter family of ACH metrics which solve the Einstein equation to infinite order and have a given three dimensional CR structure at infinity. When the parameter is 0, the metric is self-dual to infinite order. As an application, we give another proof of the fact that three dimensional CR manifolds admit CR invariant powers of the sublaplacian (CR GJMS operators) of all orders, which has been proved by Gover--Graham. We also prove the convergence of the formal solutions when the CR structure is real analytic.
\end{abstract}
\maketitle

\section{Introduction}
The GJMS operator $\mathcal{P}_{2k}$ on a conformal manifold of dimension $N$ is an  invariant linear differential operator acting on conformal densities of weight $k-N/2$ whose principal part is the power $\Delta^k$ of the Laplacian \cite{GJMS}. It plays an important role in geometric analysis on conformal manifolds, and is also related to a fundamental curvature quantity, called the $Q$-curvature, whose integral gives a global conformal invariant \cite{FG1, FH, GZ}. The GJMS operator is constructed via the (Fefferman--Graham) ambient metric \cite{FG2} or equivalently via the Poincar\'e metric whose boundary at infinity is the given conformal manifold \cite{FG1, GZ}. The ambient metric is a formal solution to the Ricci flat equation, which corresponds to the Einstein equation for the Poincar\'e metric. When the dimension $N$ is odd, the equation can be solved to infinite order and $\mathcal{P}_{2k}$ is defined for all $k\ge1$. On the other hand, when $N$ is even, an obstruction to the existence of a formal solution appears, and $\mathcal{P}_{2k}$ can only be defined for $1\le k\le N/2$ due to the ambiguity of the ambient metric at higher orders. Moreover, it is known that this result of the existence of $\mathcal{P}_{2k}$ is sharp \cite{GH}.

The CR counterpart of these operators are CR invariant powers of the sublaplacian 
$$
P_{2k}: \calE\Bigl(\frac{k-n-1}{2}, \frac{k-n-1}{2}\Bigr)\longrightarrow 
\calE\Bigl(\frac{-k-n-1}{2}, \frac{-k-n-1}{2}\Bigr)
$$
on a $(2n+1)$-dimensional CR manifold $M$, which are called the CR GJMS operators or the Gover--Graham operators \cite{GG, HPT}. Here, $\calE(w, w^\prime)$ is a complex line bundle over $M$ called the CR density of weight $(w,w^\prime)$; see \S2.1 for the definition.
One can associate a conformal structure to a circle bundle over $M$, called the Fefferman conformal structure \cite{F}, and apply the GJMS construction to produce $P_{2k}$ for $1\le k\le n+1$. Gover--Graham \cite{GG} gave more operators by using techniques of CR tractor calculus; they proved that for each $(w, w^\prime)$ such that $k=w+w^\prime+n+1\in
\mathbb{N}_+$ and $(w, w^\prime)\notin \mathbb{N}\times\mathbb{N}$, there exists a CR invariant linear differential operator $P_{w, w^\prime}: \calE(w, w^\prime)\rightarrow\calE(w-k, w^\prime-k)$ whose principal part is $\Delta_b^k$. In cases where $w=w^\prime$, these operators provide CR invariant modifications of $\Delta_b^k$ for all $k$ with $k \equiv n\ {\rm mod}\ 2$.
When $n=1$, even more operators can be constructed: CR structure is a Cartan geometry modeled on the CR sphere $S^{2n+1}=SU(n+1, 1)/P$, where $P$ is the isotropy subgroup of a point in $S^{2n+1}$, and three dimensional CR structure has a special feature from this viewpoint in that $P$ is a Borel subgroup. Then the BGG machinery developed in \cite{CSS} gives 
operators $P_{w, w^\prime}$ for $(w, w^\prime)\in\mathbb{N}\times\mathbb{N}$ when $n=1$. 
Thus one has:
\begin{thm}[{\cite[Theorem 1.3]{GG}}]\label{GJMS-three}
Let $M$ be a three dimensional strictly pseudoconvex CR manifold. For each $(w, w^\prime)$ such that $k=w+w^\prime+2\in\mathbb{N}_+$, there exists a CR invariant linear differential operator $P_{w, w^\prime}: \calE(w, w^\prime)\rightarrow\calE(w-k, w^\prime-k)$ on $M$, whose principal part is $\Delta_b^k$. 
\end{thm}

In this paper, we provide a unified proof of Theorem \ref{GJMS-three} for cases in which $w=w^\prime$. To this end, we construct an ACH (asymptotically complex hyperbolic) metric on a manifold with boundary $M$ whose Taylor expansion along $M$ is completely determined by local data of $M$. 
Our ACH metric is a refinement of the ACH Einstein metric which Matsumoto \cite{Ma1, Ma2} constructed for partially integrable CR manifolds. To state the results, let us recall some basic notions related to ACH metrics. Let $M$ be a $(2n+1)$-dimensional strictly pseudoconvex partially integrable CR manifold. Namely, $M$ has a contact distribution $H\subset TM$ together with an almost complex structure $J\in End(H)$, and the eigenspace $T^{1,0}M\subset\mathbb{C}H$ with the eigenvalue $i$ satisfies the partial integrability: $[\Gamma(T^{1,0}M), \Gamma(T^{1,0}M)]\subset \Gamma(\mathbb{C} H)$. A $\Theta$-structure on a manifold 
$\overline X$ with boundary $M$ is a conformal class $[\Theta]$ of sections $\Theta\in 
\Gamma(M, T^\ast \overline X)$ such that $\Theta|_{TM}$ is a contact form on $M$. A diffeomorphism which preserves a $\Theta$-structure is called a $\Theta$-diffeomorphism. 
On the product $M\times [0, \infty)_\r$, we define the standard $\Theta$-structure by extending 
each contact form $\th$ on $M$ to $\Theta$ so that $\Theta(\pa/\pa\r)=0$. Fix a contact form
 $\th$ on $M$ and let $\{T, Z_\a\}$ be an admissible frame. We take the local frame 
$\{\mbox{\boldmath $Z$}_\infty=\r\pa_\r, \mbox{\boldmath $Z$}_0=\r^2T, \mbox{\boldmath $Z$}_\a=\r Z_\a, \mbox{\boldmath $Z$}_{\overline{\a}}=\r Z_{\overline{\a}}\}$ 
and its dual coframe $\{\boldsymbol{\th}^\infty, \boldsymbol{\th}^0,  
\boldsymbol{\th}^\a, \boldsymbol{\th}^{\overline{\a}}\}$ on $M\times (0, \infty)_\r$. Then for any ACH metric $g$ on $X$, there exists a $\Theta$-diffeomorphism 
$\Phi: M\times [0, \infty)_\r\rightarrow\overline X$ which is defined near $M$ and restricts to the identity on $M$, such that  $\Phi^\ast g=g_{IJ}\boldsymbol{\th}^I \boldsymbol{\th}^J$ satisfies
\begin{equation*}
\begin{aligned}
g_{\infty\infty}=4, \quad g_{\infty 0}=g_{\infty\a}&=0, \quad g_{00}=1+O(\r), \quad g_{0\a}=O(\r), \\ 
 g_{\a\b}=O(\r), &\quad g_{\a\overline{\b}}=h_{\a\overline{\b}}+O(\r),
\end{aligned}
\end{equation*}
where $h_{\a\overline{\b}}$ is the Levi form on $M$. The CR manifold $M$ is called the {\it CR structure at infinity} of $g$. Matsumoto \cite{Ma1, Ma2}
 proved that for any partially integrable CR manifold $M$, there exists an ACH metric $g$ on  $M\times [0, \infty)_\r$ which satisfies 
\begin{align*}
E_{IJ}:=\Ric_{IJ}+\frac{n+2}{2}g_{IJ}=O(\r^{2n+2}), \\
\scal=-(n+1)(n+2)+O(\r^{2n+3}),
\end{align*}
where $\Ric$ is the Ricci tensor and $\scal$ is the scalar curvature. Up to the pull-back by a $\Theta$-diffeomorphism which fixes $M$, such a metric is unique modulo 
tensors which have $O(\r^{2n+2})$ coefficients and $O(\r^{2n+3})$ trace in the frame 
$\{\mbox{\boldmath $Z$}_I \}$. The order $O(\r^{2n+2})$ in the above equation is optimal in general since $(\r^{-2n-2}E_{\a\b})|_M$ is independent of the choice of a solution $g$ and defines a CR invariant tensor $\mathcal{O}_{\a\b}\in\calE_{\a\b}(-n, -n)$, called the {\it CR obstruction tensor}. Matsumoto \cite{Ma3} generalized the CR GJMS operators $P_{2k}$ to the partially integrable case by using Dirichlet-to-Neumann type operator for the eigenvalue equations of the Laplacian of $g$, but the order is again restricted to $1\le k\le n+1$ due to the presence of the obstruction.  

If we confine ourselves to the case where $M$ is an integrable CR manifold, there is a possibility to refine the construction of ACH metrics. In fact, the CR obstruction tensor vanishes for integrable CR manifolds, in particular for three dimensional CR manifolds since the CR structure is always integrable in this dimension. However, we need an additional normalization condition on the metric to ensure the uniqueness since the Einstein equation does not determine the 
$O(\r^{2n+2})$-term of the metric. A possible normalization is the K\"ahler condition; Fefferman \cite{F} constructed an approximate solution to the complex Monge--Amp\`ere equation on a strictly pseudoconvex domain $\Omega$ with boundary $M$ and defined a K\"ahler metric which satisfies $E_{IJ}=O(\r^{2n+4})$ as an ACH metric on the `square root' of $\Omega$. However, this construction also has an obstruction $\mathcal{O}\in\calE(-n-2,\, -n-2)$, called the {\it CR obstruction density}, and the metric is only determined modulo $O(\r^{2n+4})$. 

In this paper, we show that the self-dual equation $W^{-}=0$ works as a better normalization when $M$ is three dimensional. The anti self-dual part $W^{-}$ of the Weyl curvature is connected to the Ricci tensor by the Bianchi identity 
\begin{equation}\label{bianchi}
\nabla^I W^{-}_{IJKL}=C^{-}_{JKL},
\end{equation}
where $C^{-}_{IJK}$ is the anti self-dual part of the Cotton tensor $C_{IJK}$, which is defined by $C_{IJK}:=\nabla_{K}P_{IJ}-\nabla_{J}P_{IK}$ with the Schouten tensor $P_{IJ}=\frac{1}{2} \Ric_{IJ}-\frac{1}{12} \scal\, g_{IJ}$.
It follows from \eqref{bianchi} that the equation $E_{IJ}=O(\r^4)$ implies $W^{-}_{IJKL}=O(\r^4)$, and it turns out that the further normalization $W^{-}_{IJKL}=O(\r^5)$ determines $g_{IJ}$ modulo $O(\r^5)$. We can then solve 
$E_{IJ}=O(\r^6)$, which implies $W^{-}_{IJKL}=O(\r^6)$. In the next step, besides the Einstein equation, we have freedom to prescribe the value of 
$$
\eta:=\bigl(\r^{-6}W^{-}_{\infty0\infty0}\bigr)\big|_M.
$$
If the Taylor coefficients of $g_{IJ}$ along $M$ have universal expressions in terms of pseudo-hermitian structure, $\eta$ defines a CR invariant of weight $(-3, -3)$ (see Lemma \ref{eta}). Thus, we should prescribe $\eta$ to be a CR invariant in order to obtain a CR invariant normalization condition. It is known that a CR invariant in $\calE(-3, -3)$ on a three dimensional CR manifold is unique up to a constant multiple \cite{G}, so there is no choice but to set $\eta=\lambda\mathcal{O}$ with a constant $\lambda\in\R$. After this step, the Einstein equation determines $g_{IJ}$ to infinite order, and in the case $\lambda=0$ the self-duality follows automatically from \eqref{bianchi}. Thus our main theorem reads as follows:

\begin{thm}\label{main-thm}
Let $M$ be a three dimensional strictly pseudoconvex CR manifold, and let $\lambda\in\R$.
Then there exists an ACH metric $g^{\lambda}_{IJ}$ on $M\times[0, \infty)_\r$ which has $M$ as the CR structure at infinity and satisfies
\[
\Ric_{IJ}+\frac{3}{2}g^\lambda_{IJ}=O(\r^\infty), \quad
W^-_{IJKL}=O(\r^6), \quad
 \eta=\lambda\mathcal{O},
\]
where $\eta$ is the density defined by \eqref{def-eta}. The metric $g^\lambda_{IJ}$ is unique modulo $O(\r^\infty)$ up to the pull-back by a $\Theta$-diffeomorphism which fixes $M$. Moreover, $g^0_{IJ}$ satisfies $W^-_{IJKL}=O(\r^\infty)$.
\end{thm}

The Taylor coefficients of $g^\lambda_{IJ}$ along the boundary have universal expressions in terms of the pseudo-hermitian structure for a fixed contact form.

By applying the construction of the CR GJMS operators via the ACH metric \cite{Ma3}, we obtain the following theorem, which is a special case of Theorem \ref{GJMS-three}:

\begin{thm}\label{GJMS}
Let $M$ be a three dimensional strictly pseudoconvex CR manifold, and let $\lambda\in\R$. Then, there exists a CR invariant linear differential operator
$$
P^\lambda_{2k}: \calE(k/2-1,\ k/2-1)\longrightarrow 
\calE(-k/2-1,\ -k/2-1)
$$
which is a polynomial in $\lambda$ of degree $\le k/3$, and has the principal part $\Delta_b^k$.  
\end{thm}

Let us mention a similar construction in conformal geometry. Fefferman--Graham \cite{FG2} constructed a formal solution to the self-dual Einstein equation for the Poincar\'e metric with a given three dimensional conformal manifold $\mathcal{M}$ as its conformal infinity. Thus our result is a CR analogue of their construction. When $\mathcal{M}$ is real analytic, LeBrun \cite{LeB} showed by twistor methods that there exists a real analytic self-dual Einstein metric on $\mathcal{M}\times(0, \e)$ with the conformal infinity $\mathcal{M}$. The metric of Fefferman--Graham gives the Taylor expansion of LeBrun's metric. In CR case, Biquard \cite{B} showed the existence of a self-dual Einstein ACH metric with a given real analytic CR three-manifold as its infinity by using twistor methods.
Thus our formal solution $g^0_{IJ}$ gives the Taylor expansion of Biquard's metric. In this paper, we 
prove the convergence of $g^\lambda_{IJ}$ by applying  the results of Baouendi--Goulaouic \cite{BG} on singular nonlinear Cauchy problems.

\begin{thm}\label{convergence}
Suppose  $M$ is a real analytic strictly pseudoconvex CR manifold of dimension three. Then the formal solution  $g^\lambda_{IJ}$ in Theorem \ref{main-thm} converges to a real analytic ACH metric near $M$. 
\end{thm}

This paper is organized as follows. In \S2, we review pseudo-hermitian geometry on a CR manifold and 
basic notions on ACH metric. By following Matsumoto \cite{Ma1}, we describe the Levi-Civita connection of an ACH metric in terms of the extended Tanaka--Webster connection. In \S3, we clarify the relationship between the Einstein equation and the self-dual equation, and compute the variation of curvature quantities under a perturbation of the metric. \S4 is devoted to the proof of Theorem \ref{main-thm}; we construct a one parameter family of formal solutions to the Einstein equation and examine the dependence on the parameter. Then, in \S5 we use these metrics to construct the CR GJMS operators and prove Theorem \ref{GJMS}. Finally, in \S6 we show the convergence of the formal solutions in case $M$ is a real analytic 
CR manifold.
\ \\

\noindent{\bf Acknowledgements}.
The author is grateful to Kengo Hirachi and Yoshihiko Matsumoto for invaluable comments on the results. 
He also thanks Matsumoto for checking the author's suggestion of some minor corrections to the computation in his thesis \cite{Ma1}.

\section{CR structure and ACH metric}
\subsection{Pseudo-hermitian geometry}
Let $M$ be a $(2n+1)$-dimensional $C^\infty$ manifold.  A pair $(H, J)$ is called a {\it CR structure} 
on $M$ if $H$ is a rank $2n$ subbundle of $TM$ and $J$ is an almost complex structure on $H$ which satisfies the (formal) integrability condition
$$
[\Gamma(T^{1,0}M), \Gamma(T^{1,0}M)]\subset\Gamma(T^{1,0}M),
$$ 
where $T^{1,0}M\subset\mathbb{C}H$ is the eigenspace of $J$ with the eigenvalue $i$. We note that the integrability condition automatically holds when $M$ is three dimensional. For any real 1-form $\th$ such that ${\rm Ker}\, \th=H$, we define the {Levi form} $h_{\th}$ by
$$
h_\th(Z, \overline W)=-id\th(Z, \overline W)
$$
for $Z, W\in T^{1,0}M$. We say the CR structure is {\it strictly pseudoconvex} if $h_\th$ is positive definite for some $\th$. Since $h_{f\th}=fh_\th$ for any function $f$, such $\th$ is determined up to a multiple by a positive function. When $M$ is strictly pseudoconvex, $H$ defines a contact structure, so we call $\th$ a contact form. The {\it Reeb vector field} is the real vector field $T$ uniquely determined by the conditions
$$
\th(T)=1, \quad T\lrcorner\, d\th=0.
$$ 
Let $\{Z_\a\}$ be a local frame for $T^{1, 0}M$. If we put $Z_{\overline \a}:=\overline{Z_\a}$, then $\{T, Z_\a, Z_{\overline \a}\}$ gives a local frame for $\mathbb{C}TM$, which we call an {\it admissible frame}. The dual coframe $\{\th, \th^\a, \th^{\overline \a}\}$ is called an {\it admissible coframe} and satisfies
$$
d\th=ih_{\a\overline\b}\th^\a\wedge\th^{\overline \b},
$$
where $h_{\a\overline\b}=h_\th(Z_\a, Z_{\overline\b})$.

The {\it CR canonical bundle} is defined by $K_M:=\wedge^{n+1}(T^{0,1}M)^\perp\subset\wedge^{n+1}\mathbb{C}T^\ast M$, where $T^{0,1}M:=\overline{T^{1,0}M}$. When $K_M^{-1}$ admits an 
$(n+2)$-nd root $\calE(1,0)$, the {\it CR density bundle} is defined by
\begin{equation}\label{CR-density}
\calE(w, w^\prime)=\calE(1,0)^{\otimes w}\otimes\overline{\calE(1,0)}{}^{\otimes w^\prime}
\end{equation}
for each $(w, w^\prime)\in\mathbb{C}^2$ with $w-w^\prime\in\mathbb{Z}$. In this paper, we restrict ourselves to the cases $w=w^\prime$. In these cases, the definition \eqref{CR-density} is independent of the choice of $\calE(1,0)$ so we can define $\calE(w, w)$ without assuming the global existence of $\calE(1,0)$.  We also denote the space of sections of these bundles by the same symbols, and call them {\it CR densities}.

For any contact form $\th$, there exists a local nonvanishing section $\zeta$ of $K_M$, unique up to a multiple of a $U(1)$-valued function, which satisfies
$$ 
\th\wedge(d\th)^n=i^{n^2}n!\th\wedge(T\lrcorner\,\zeta)\wedge(T\lrcorner\,\overline\zeta).
$$
Then, the weighted contact form $\bth:=\th\otimes |\zeta|^{-2/(n+2)}\in\Gamma(T^\ast M\otimes\calE(1,1))$ is defined globally and  independent of the choice of $\th$. Thus, there is a one to one correspondence between the set of  contact forms and the set of positive sections $\tau\in\calE(1,1)$, called {\it CR scales}. We define the CR invariant weighted Levi form $\bh_{\a\overline\b}:=\tau h_{\a\overline\b}$ by putting a weight to $h_\th$ with the CR scale $\tau$ corresponding to $\th$. We raise and lower the indices of tensors on $\mathbb{C}H$ by $\bh_{\a\overline\b}$ and its inverse $\bh^{\a\overline\b}$, which has weight $(-1, -1)$.

For a fixed contact form $\th$, we can define a canonical linear connection $\nabla$ on $TM$, called the {\it Tanaka--Webster connection}. It preserves $T^{1,0}M$ and satisfies $\nabla T=0$, $\nabla h_\th=0$. In an admissible frame 
$\{T, Z_\a, Z_{\overline \a}\}$, the connection 1-forms $\omega_\b{}^\a$ satisfy the structure equation
$$
d\th^\a=\th^\b\wedge\omega_\b{}^\a+A^\a{}_{\overline \b}\,\bth\wedge\th^{\overline \b}.
$$ 
The tensor $A_{\a\b}:=\overline{A_{\overline\a\overline\b}}$ satisfies $A_{\a\b}=A_{\b\a}$ and is called the {\it Tanaka--Webster torsion tensor}. 
We use the index 0 for the direction of $T$, and we denote the components of covariant derivatives of a tensor by indices preceded by a comma, e.g., $A_{\a\g, \overline\b}=\nabla_{\overline\b} A_{\a\g}$. We omit the comma for covariant derivatives of a function. 
The curvature form $\Omega_{\a}{}^\b=d\omega_{\a}{}^\b-\omega_{\a}{}^\g\wedge\omega_\g{}^\b$ 
is given by
\begin{equation}\label{curvature-TW}
\begin{aligned}
\Omega_\a{}^\b&=R_\a{}^\b{}_{\g\overline\mu} \th^\g\wedge\th^{\overline \mu}+A_{\a\g,}{}^\b\th^\g\wedge\bth-A^\b{}_{\overline \g,\a}\th^{\overline \g}\wedge\bth \\
& \quad -iA_{\a\g}\th^\g\wedge\th^\b+i\bh_{\a\overline\g}A^\b{}_{\overline\mu}
\th^{\overline\g}\wedge\th^{\overline\mu}.
\end{aligned}
\end{equation}
The tensor $R_\a{}^\b{}_{\g\overline\mu}$ is called the {\it Tanaka--Webster curvature tensor}. Taking traces with the weighted Levi form, we define the Tanaka--Webster Ricici tensor $\Ric_{\a\overline\b}:=
R_\g{}^\g{}_{\a\overline\b}$ and the Tanaka--Webster scalar curvature $\scal:=\Ric_\a{}^\a$.
The {\it sublaplacian} is the differential operator $\Delta_b: \calE(w, w^\prime)\rightarrow
\calE(w-1, w^\prime-1)$ defined by 
$$
\Delta_b f=-\bh^{\a\overline\b}(\nabla_\a\nabla_{\overline\b}+\nabla_{\overline\b}\nabla_\a)f.
$$
If we rescale the contact form as $\wh\th=e^\U\th$, the Tanaka--Webster connection and its  curvature quantities satisfy transformation formulas involving the derivatives of the scaling factor $\U$; see e.g.,  \cite{L}.  We note that in dimension three the rank of $T^{1,0}M$ is 1 and the curvature form \eqref{curvature-TW} is reduced to 
$$
\Omega_1{}^1=\scal\,\bh_{1\overline1}\th^1\wedge\th^{\overline1}+A_{11,}{}^1\th^1\wedge\bth-
A^1{}_{\overline1, 1}\th^{\overline1}\wedge\bth.
$$
Also, in this dimension, $M$ is locally CR diffeomorphic to the standard sphere $S^3$ if and only if the {\it Cartan tensor}
$$
Q_{11}:=\frac{1}{6}\scal_{11}+\frac{i}{2}\scal \,A_{11}-A_{11,0}-\frac{2i}{3}A_{11, \overline1}{}^{\overline1}
$$
vanishes identically. The Cartan tensor is a CR invariant tensor of weight $(-1, -1)$. We also have a CR invariant density defined by 
\begin{equation}\label{def-O}
\mathcal{O}:=(\nabla^1\nabla^1-iA^{11})\,Q_{11}\in\calE(-3, -3),
\end{equation}
called the {\it obstruction density}. It follows from the Bianchi identity for the Cartan tensor that 
$\mathcal{O}$ is a real density \cite{CL}. There is also a CR invariant density, called the  obstruction density, on higher dimensional CR manifolds and it appears as the logarithmic coefficient in the asymptotic expansion of the solution to the complex Monge--Amp\`ere equation on strictly pseudoconvex domain \cite{LM}. In dimension three, a CR invariant of weight $(-3, -3)$ is unique up to a constant multiple \cite{G}, so it is necessarily a multiple of $\mathcal{O}$.

\subsection{ACH metrics}
The ACH metric was introduced by Epstein--Melrose--Mendoza \cite{EMM} as a generalization of the complex hyperbolic metric on the ball. In this paper, we define it by using the characterization via
 the normal form.

Let $X$ be the interior of a $(2n+2)$-dimensional $C^\infty$ manifold whose boundary $M$ is  equipped with a strictly pseudoconvex CR structure $(H, J)$. A conformal class $[\Theta]$ in 
$\Gamma(M, T^\ast\overline{X})$ is called a $\Theta$-{\it structure} if $\Theta|_{TM}$ gives a contact form on $M$ for each $\Theta\in[\Theta]$. We call $(\overline X, [\Theta])$ a 
$\Theta$-{\it manifold}. Let $(\overline{X}^\prime, [\Theta^\prime])$ be another $\Theta$-manifold with the same boundary $M$. Then, a diffeomorphism $\Phi$ from a neighborhood of $M$ in $\overline X$ to a neighborhood of $M$ in $\overline{X}^\prime$ is called a $\Theta$-{\it diffeomorphism} if it fixes $M$ and 
satisfies $[\Phi^\ast\Theta^\prime]=[\Theta]$. We take a boundary defining function $\r\in C^\infty(\overline X)$ which is positive on $X$. A vector field $V$ on $\overline X$ is called a 
$\Theta$-{\it vector field} if it satisfies 
$$
V|_M=0, \quad \wt\Theta(V)=O(\r^2),
$$ 
where $\wt\Theta$ is an arbitrary extension of a $\Theta\in[\Theta]$. Note that the definition is independent of the choice of  $\Theta$ and $\wt\Theta$. We extend $\{d\r, \wt\Theta\}$ to a local coframe $\{d\r, \wt\Theta, \a^1, \dots, \a^{2n}\}$ for $T^\ast\overline X$ near $M$. Let 
$\{N, T, Y_1, \dots, Y_{2n}\}$ be the dual frame. Then, any $\Theta$-vector field $V$ can be written as
$$
V=V^{\infty}(\r N)+V^0 (\r^2T)+V^i (\r Y_i), \quad V^\infty, V^0, V^i\in C^\infty(\overline X).
$$
If we take another local coframe $\{d\r^\prime, \wt\Theta^\prime, \a^\prime{}^i\}$ and its dual $\{N^\prime, T^\prime, Y^\prime_i\}$, then the transition function between $\{\r N, \r^2T, \r Y_i\}$ and $\{\r^\prime N^\prime, \r^\prime{}^2T^\prime, \r^\prime Y^\prime_i\}$ is smooth and nondegenerate up to $M$, so there exists a vector bundle $^\Theta T\overline X$ over $\overline X$ for which 
$\{\r N, \r^2T, \r Y_i\}$ gives a local frame. A $\Theta$-vector field is identified with a section of this bundle and we call $^\Theta T\overline X$ the $\Theta$-{\it tangent bundle}. A fiber metric on $^\Theta T\overline X$ is called a $\Theta$-{\it metric}. Since the restriction $^\Theta T\overline X|_X$ is canonically isomorphic to $TX$, a $\Theta$-metric defines a Riemannian metric on $X$.  A local frame $\{\mbox{\boldmath $Z$}_I\}$ of $^\Theta T\overline X$ is called a $\Theta$-{\it frame}.  We also consider the dual $^\Theta T^\ast \overline X$ of the $\Theta$-tangent bundle and various tensor bundles, whose sections are called $\Theta$-{\it tensors}. A $\Theta$-tensor is said to be $O(\r^m)$ if each component in a $\Theta$-frame is $O(\r^m)$.
$\Theta$-vector fields are closed under the Lie bracket, and those which vanish at a fixed point $p\in M$ form an ideal. Thus the fiber $^\Theta T_p\overline X$ becomes a Lie algebra, which we call the {\it tangent algebra}.  

The product $M\times [0, \infty)_\r$ has a canonical $\Theta$-structure, called the 
{\it standard} $\Theta$-{\it structure}, which is defined by extending each contact form $\th$ on $M$ to $\Theta\in\Gamma(M, T^\ast \overline X)$ with $\Theta(\pa/\pa \r)=0$. Let $\th$ be a contact form and $\{T, Z_\a, Z_{\overline \a}\}$ an admissible frame for $\mathbb{C}TM$. We extend 
$\{T, Z_\a, Z_{\overline \a}\}$ to $M\times [0, \infty)_\r$ in the trivial way, and define a 
(complexified) $\Theta$-frame $\{\mbox{\boldmath $Z$}_I\}$ by
$$
\mbox{\boldmath $Z$}_\infty=\r\pa_\r, \quad \mbox{\boldmath $Z$}_0=\r^2 T, 
\quad \mbox{\boldmath $Z$}_\a=\r Z_\a, \quad \mbox{\boldmath $Z$}_{\overline \a}=\r Z_{\overline \a},
$$
where $\pa_\r=\pa/\pa \r$. A $\Theta$-metric $g$ on $M\times [0, \infty)_\r$ is called a 
{\it normal form ACH metric} if the components $g_{IJ}=g(\mbox{\boldmath $Z$}_I,
\mbox{\boldmath $Z$}_J)$ satisfy
\begin{equation}\label{normal}
\begin{aligned}
g_{\infty\infty}=4, \quad g_{\infty 0}=g_{\infty\a}&=0, \quad g_{00}=1+O(\r), \quad g_{0\a}=O(\r), \\ 
 g_{\a\b}=O(\r), &\quad g_{\a\overline{\b}}=h_{\a\overline{\b}}+O(\r),
\end{aligned}
\end{equation}
where $h_{\a\overline{\b}}=h_\th(Z_\a, Z_{\overline\b})$.
On a general $\Theta$-manifold $(\overline X, [\Theta])$, the ACH metric is defined as follows:

\begin{dfn}
A $\Theta$-metric $g$ on $\overline X$ is called an ACH metric if for any contact form $\th$ on $M$, there exist a neighborhood $U\subset\overline X$ of $M$ and a $\Theta$-diffeomorphism $\Phi_\th: M\times [0, \infty)_\r\rightarrow U$ such that $\Phi_\th^\ast g$
is a normal form ACH metric.
\end{dfn}

We remark that there is an alternative definition of the ACH metric which involves only boundary value of $g$; see \cite[Definition 4.6]{Ma1}. 

The germ of $\Phi_\th$ along $M$ is unique, and we call $\r\circ\Phi_\theta^{-1}$ the {\it model defining function} for $\th$. We identify a neighborhood of $M$ in $\overline X$ with $M\times[0, \e)_\r$ 
through $\Phi_\th$ and regard $\{\mbox{\boldmath $Z$}_I\}$ as a $\Theta$-frame on $\overline X$.
The following proposition will be used in the proof of Lemma \ref{eta}.

\begin{prop} \label{Z-M}
The boundary values $\mbox{\boldmath $Z$}_\infty|_M, \mbox{\boldmath $Z$}_0|_M$ are  independent of $\th$ and determined only by the ACH metric $g$.
\end{prop}
\begin{proof}
By strict pseudoconvexity of $(H, J)$, the derived Lie algebras of the tangent algebra $^\Theta T_p\overline X$ at a point $p\in M$ are given by  
\begin{align*}
\mathcal{D}^1&:=[^\Theta T_p\overline X,{}^\Theta T_p\overline X]
={\rm span}\{(\mbox{\boldmath $Z$}_0)_p, (\mbox{\boldmath $Z$}_\a)_p, 
(\mbox{\boldmath $Z$}_{\overline \a})_p\}, \\
\mathcal{D}^2&:=[\mathcal{D}^1, \mathcal{D}^1]={\rm span}\{(\mbox{\boldmath $Z$}_0)_p\}.
\end{align*}
Thus, $(\mbox{\boldmath $Z$}_\infty)_p$ and $(\mbox{\boldmath $Z$}_0)_p$ are oriented basis of $(\mathcal{D}^1)^\perp$ and $\mathcal{D}^2$ respectively. Since they are normalized by ${|(\mbox{\boldmath $Z$}_\infty)_p|}_g^2=4$, $|(\mbox{\boldmath $Z$}_0)_p|_g^2=1$, they are independent of $\th$.
\end{proof}

Let $\th, \wh\th=e^\U\th$ be contact forms on $M$ and $\r, \wh\r$ the corresponding model defining functions. Then there exists a positive function $f$ on $\overline X$ such that $\wh\r=f\r$. Since the Reeb vector fields are related as $\wh T=e^{-\U}(T-ih^{\a\overline\g}\U_{\overline \g}Z_\a+ih^{\g\overline\a}\U_\g Z_{\overline\a})$, we have
$$
\wh{\mbox{{\boldmath $Z$}}}_0=\wh{\r}^{\,2}\wh T=e^{-\U}f^2\mbox{\boldmath $Z$}_0+O(\r)
$$
as a $\Theta$-vector field, where we regard $\U$ as a function on a neighborhood of $M$.
It follows from $\wh{\mbox{\boldmath $Z$}}_0|_M=\mbox{\boldmath $Z$}_0|_M$ that $f|_M=e^{\U/2}$. Thus we have
\begin{equation}\label{rho}
\wh\r=e^{\U/2}\r+O(\r^2).
\end{equation}
In particular, a contact form is recovered from the 1-jet of the corresponding model defining function along the 
boundary.

\subsection{The Levi-Civita connection}
Let $g$ be an ACH metric on a $\Theta$-manifold $(\overline X, [\Theta])$ with boundary $M$. Here and after, we assume that $M$ is three dimensional. We lower and raise the indices of $\Theta$-tensors by $g_{IJ}$ and its inverse $g^{IJ}$. In order to describe the Levi-Civita connection of $g$, we introduce an extension of the Tanaka--Webster connection by following \cite{Ma1, Ma2}; we refer the reader to \cite[\S 6.2]{Ma1} or \cite[\S 4]{Ma2} for a more detailed exposition. 

Let $\th$ be a contact form on $M$. We identify a neighborhood of $M$ in $\overline X$ with $M\times [0, \e)_\r$ by the $\Theta$-diffeomorphism determined by $\th$. We take an admissible frame $\{T, Z_1, Z_{\overline 1}\}$ and define the {\it extended Tanaka--Webster connection} $\overline{\nabla}$ on $T\overline X$ by 
\begin{gather*}
\overline{\nabla}\pa_\r=0, \quad\overline{\nabla}_{\pa_\r}T=\overline{\nabla}_{\pa_\r}Z_1=0, \\
\overline{\nabla}_{T}Z_1=\nabla^{\rm TW}_T Z_1, \quad \overline{\nabla}_{Z_1}Z_1=\nabla^{\rm TW}_{Z_1}Z_1,
\quad \overline{\nabla}_{Z_{\ol1}}Z_1=\nabla^{\rm TW}_{Z_{\ol1}}Z_1,
\end{gather*}
where $\nabla^{\rm TW}$ denotes the Tanaka--Webster connection associated with $\th$. Then, $\overline{\nabla}$ is a $\Theta$-{\it connection} in the sense that if $V, W$ are $\Theta$-vector fields, so is 
the covariant derivative $\overline{\nabla}_V W$. We take the $\Theta$-frame $\{\mbox{\boldmath $Z$}_I\}=\{\r\pa_\r, \r^2T, \r Z_1, \r Z_{\overline1}\}$ and define the Christoffel symbols $\overline{\Gamma}_{IJ}{}^K$ by 
$\overline{\nabla}_{\scriptsize{\mbox{\boldmath $Z$}_I} }\mbox{\boldmath $Z$}_J= 
\overline{\Gamma}_{IJ}{}^K \mbox{\boldmath $Z$}_K$. A simple calculation shows that
\begin{equation}\label{Gamma-bar}
\begin{aligned}
&\overline{\Gamma}_{\infty\infty}{}^\infty=1,  &  &\overline{\Gamma}_{\infty0}{}^0=2, &  &\overline{\Gamma}_{\infty 1}{}^1=1, \\
&\overline{\Gamma}_{01}{}^1=\r^2\Gamma_{01}{}^1,&  &\overline{\Gamma}_{11}{}^1=\r\Gamma_{11}{}^1, & 
&\overline{\Gamma}_{\overline1 1}{}^1=\r\Gamma_{\overline1 1}{}^1,
\end{aligned}
\end{equation}
where $\Gamma_{ij}{}^k$ are the Christoffel symbols of $\nabla^{\rm TW}$ with respect to $\{T, Z_1, Z_{\overline1}\}$; the components which cannot be obtained by taking complex conjugates of \eqref{Gamma-bar} are 0. It follows from \eqref{Gamma-bar} that the components of the covariant derivative of a $\Theta$-tensor $S_{I_1\cdots I_p}{}^{J_1\cdots J_q}$ are computed as
\begin{equation}\label{nabla-bar}
\begin{aligned}
&\overline{\nabla}_\infty S_{I_1\cdots I_p}{}^{J_1\cdots J_q}=\bigl(\r\pa_\r-\#(I_1\cdots I_p)+\#(J_1\cdots J_q)\bigr)S_{I_1\cdots I_p}{}^{J_1\cdots J_q}, \\
&\overline{\nabla}_0 S_{I_1\cdots I_p}{}^{J_1\cdots J_q}=\r^2\nabla^{\rm TW}_0 S_{I_1\cdots I_p}{}^{J_1\cdots J_q}, 
\\
&\overline{\nabla}_1 S_{I_1\cdots I_p}{}^{J_1\cdots J_q}=\r\nabla^{\rm TW}_1 S_{I_1\cdots I_p}{}^{J_1\cdots J_q},
\end{aligned}
\end{equation}
where $\#(I_1\cdots I_p):=p+({\rm the\ number\ of\ 0}s)$ and we regard $S$ as a tensor on $\mathbb{C}H$ when we apply $\nabla^{\rm TW}$ to it (\cite[Lemma 6.2]{Ma1}, \cite[(4.9)]{Ma2}). 

The torsion tensor $\overline T_{IJ}{}^K$ and the curvature tensor $\overline R_I{}^J{}_{KL}$ of $\overline{\nabla}$ are defined by  
\begin{gather*}
(\overline{\nabla}_V W-\overline{\nabla}_W V-[V, W])^K=\overline T_{IJ}{}^K V^I W^J,\\
(\overline{\nabla}_V\overline{\nabla}_W Y-\overline{\nabla}_W\overline{\nabla}_V Y
-\overline{\nabla}_{[V, W]}Y)^J=\overline R_I{}^J{}_{KL}Y^I V^K W^L
\end{gather*}
 respectively. In the $\Theta$-frame $\{\mbox{\boldmath $Z$}_I\}$, the components are given by 
\begin{equation}\label{torsion}
\overline T_{1\overline1}{}^0=i h_{1\overline1}, \quad \overline T_{01}{}^{\overline1}=\r^2 A_1{}^{\overline1},
\end{equation}
and 
\begin{equation}\label{curvature}
\overline R_1{}^1{}_{1\overline1}=\r^2\scal^{\rm TW} h_{1\overline1}, \quad
\overline R_1{}^1{}_{01}=-\r^3A_{11,}{}^1, \quad
\overline R_1{}^1{}_{0\overline1}=\r^3 A^1{}_{\overline1,}{}_{1},
\end{equation}
where $\scal^{\rm TW}$ denotes the Tanaka--Webster scalar curvature, and we have removed the CR weights  in the Tanaka--Webster tensors by the CR scale corresponding to $\th$. The components which cannot be 
obtained from \eqref{torsion}, \eqref{curvature} by the symmetries of $\overline T$,
$\overline R$ or by taking the complex conjugates are all 0. The nonzero components of the Ricci tensor $\ol R_{IJ}=\ol R_I{}^K{}_{KJ}$ are given by 
$$
\ol R_{1\ol1}=\r^2\scal^{\rm TW} h_{1\ol1}, \quad \ol R_{10}=\r^2 A_{11,}{}^1.
$$

Let $\nabla$ be the Levi-Civita connection of $g$, which is also a $\Theta$-connection 
(\cite[Proposition 4.4]{Ma1}). We define the {\it difference} $\Theta$-{\it tensor} $D_{IJ}{}^K$ by
$$
\nabla_I V^K=\overline{\nabla}_I V^K+D_{IJ}{}^K V^J.
$$
Since $\nabla$ is torsion-free, we have 
\begin{equation}\label{comm-D}
D_{IJ}{}^K=D_{JI}{}^K+\overline T_{JI}{}^K.
\end{equation}
Using this relation and the fact $\nabla g=0$, we obtain
\begin{equation}\label{D-formula}
2D_{IJK}=\overline{\nabla}_I g_{JK}+\overline{\nabla}_J g_{KI}-\overline{\nabla}_K g_{IJ}-\overline T_{IJK}
+\overline T_{JKI}-\overline T_{KIJ}.
\end{equation}
We will compute $D_{IJ}{}^K$ by these formulas. Since the components $g_{IJ}$ satisfy \eqref{normal}, $g$ is described by $\r$-dependent tensors 
$\varphi_{ij}$ on $M$ defined by
$$
g_{00}=1+\varphi_{00}, \quad g_{01}=\varphi_{01}, \quad g_{11}=\varphi_{11}, \quad g_{1\overline1}=h_{1\overline1}
+\varphi_{1\overline1}.
$$
In the construction of a formal solution to the self-dual Einstein equation, we need to examine the effect of a perturbation
\begin{equation}\label{perturb}
\varphi_{ij} \longmapsto \varphi_{ij}+\psi_{ij}, \quad \psi_{ij}=O(\r^m)
\end{equation}
on the curvature quantities of $g$. Then it is useful in the computation to ignore irrelevant terms on which the perturbation causes only changes in higher orders.  Such terms are of the form
\begin{equation}\label{neg}
O(\r)\cdot (\r\pa_\r)^l \mathcal{D} \varphi_{ij},
\end{equation}
where $\mathcal{D}$ is a $\r$-dependent differential operator on $M$. These are called {\it negligible terms}. 
In fact, a negligible term changes by $O(\r^{m+1})$ under the perturbation \eqref{perturb}.
Thus, it suffices to compute $D_{IJ}{}^K$ modulo negligible terms. For simplicity, we assume that the admissible frame $\{Z_1\}$ is unitary with respect to the Levi form; namely $h_{1\ol1}=1$. Noting that $\varphi_{ij}=O(\r)$, we have 
\begin{equation}\label{g-inverse}
\begin{aligned}
g^{\infty\infty}=\frac{1}{4}, \quad g^{\infty0}=g^{\infty1}=0, \quad g^{00}\equiv1-\varphi_{00}, \\
 g^{1\ol1}\equiv-\varphi_{1\ol1}, \quad g^{11}\equiv-\varphi_{\ol1\ol1}\qquad\qquad\quad
\end{aligned}
\end{equation}
modulo negligible terms. By computing with \eqref{nabla-bar}, \eqref{torsion}, \eqref{comm-D}, \eqref{D-formula}, \eqref{g-inverse}
we obtain the following result: 

\begin{lem}[{\cite[Lemma 6.4]{Ma1}}, {\cite[Table 1]{Ma2}}]\label{D}
Let $\{T, Z_1, Z_{\ol1}\}$ be a unitary admissible frame and $\{\mbox{\boldmath $Z$}_I\}=\{\r\pa_\r, \r^2T, \r Z_1, \r Z_{\overline1}\}$ the associated $\Theta$-frame. Then, modulo negligible terms, the components $D_{IJ}{}^K$ are given by
\begin{align*}
&D_{\infty\infty}{}^\infty \equiv-1,  & &D_{\infty0}{}^\infty\equiv D_{\infty1}{}^\infty\equiv0,  & \\
\displaystyle 
&D_{00}{}^\infty\equiv\frac{1}{2}-\frac{1}{8}(\r\pa_\r-4)\varphi_{00},  & 
\displaystyle 
&D_{01}{}^\infty\equiv-\frac{1}{8}(\r\pa_\r-3)\varphi_{01}, \\
\displaystyle 
&D_{1\ol1}{}^\infty\equiv \frac{1}{4}- \frac{1}{8}(\r\pa_\r-2)\varphi_{1\ol1}, &
\displaystyle 
&D_{11}{}^\infty\equiv -\frac{1}{8}(\r\pa_\r-2)\varphi_{11}, \\
&D_{\infty\infty}{}^1\equiv D_{00}{}^1\equiv D_{\ol1\ol1}{}^1\equiv0,  & \displaystyle 
&D_{\infty\ol1}{}^1\equiv \frac{1}{2}\r\pa_\r \varphi_{\ol1\ol1}, &
\displaystyle 
&D_{0\ol1}{}^1\equiv \frac{i}{2}\varphi_{\ol1\ol1},  \\
\displaystyle 
&D_{\infty0}{}^1\equiv \frac{1}{2}(\r\pa_\r+1)\varphi_{0\ol1}, &
\displaystyle 
&D_{01}{}^1\equiv \frac{i}{2}(1+\varphi_{00}-\varphi_{1\ol1}), & 
\displaystyle 
&D_{1\ol1}{}^1\equiv\frac{i}{2}\varphi_{0\ol1}, \\
\displaystyle 
&D_{\infty1}{}^1\equiv -1+\frac{1}{2}\r\pa_\r\varphi_{1\ol1}, &  
\displaystyle 
&D_{\ol1 0}{}^1\equiv \frac{i}{2}\varphi_{\ol1\ol1}+\r^2 A_{\ol1\ol1},  &
\displaystyle 
&D_{11}{}^1\equiv i\varphi_{01}, \\
&D_{\infty\infty}{}^0 \equiv D_{00}{}^0\equiv0, & 
\displaystyle 
&D_{1\ol1}{}^0 \equiv -\frac{i}{2}, &
\displaystyle 
&D_{11}{}^0 \equiv -\r^2 A_{11}, \\
\displaystyle 
&D_{\infty0}{}^0 \equiv -2+\frac{1}{2}\r\pa_\r \varphi_{00}, &
\displaystyle 
&D_{\infty1}{}^0 \equiv\frac{1}{2}(\r\pa_\r-1)\varphi_{01}, &
\displaystyle 
&D_{01}{}^0 \equiv-\frac{i}{2}\varphi_{01}. 
\end{align*}
The components which are not displayed are obtained by taking the complex conjugates or using  
the relation \eqref{comm-D}.
\end{lem}

\begin{rem}
We have modified a typographical error in {\cite[Table 6.2]{Ma1}}, {\cite[Table 1]{Ma2}}; the value of $D_{01}{}^1$ above differs by $-\frac{i}{2} \varphi_{1\ol1}$ from that in \cite{Ma1, Ma2}. (Note that $D_{IJ}{}^K$ is denoted by $D^K{}_{IJ}$ in \cite{Ma1} and by $D_J{}^K{}_I$ in \cite{Ma2}.) The correct value is used in the other computations in \cite{Ma1, Ma2}.
\end{rem}

\section{The self-dual Einstein equation}

Let $g$ be an ACH metric on a four dimensional $\Theta$-manifold $(\ol X, [\Theta])$ which has a strictly pseudoconvex CR manifold $M$ as its boundary. We fix a contact form $\th$ on $M$ and identify a neighborhood of $M$ as $M\times [0, \e)_\r$, where $\r$ is the model defining function for $\th$. We take a unitary admissible frame $\{T, Z_1, Z_{\ol1}\}$ on $M$ and work in the associated $\Theta$-frame $\{\mbox{\boldmath $Z$}_I\}=\{\r\pa_\r, \r^2T, \r Z_1, \r Z_{\overline1}\}$.

\subsection{The Einstein equation}
We will recall from \cite{Ma1, Ma2} the computation of the Einstein tensor modulo negligible terms which is needed in the construction of the Einstein ACH metric. We set
$$
E_{IJ}:=\Ric_{IJ}+\frac{3}{2}g_{IJ}.
$$
In terms of the extended Tanaka--Webster connection and the difference $\Theta$-tensor, the curvature tensor of $g$ is expressed as
\begin{equation}\label{curvature-g}
\begin{aligned}
R_I{}^J{}_{KL}&=\ol R_I{}^J{}_{KL}+\ol\nabla_K D_{LI}{}^J-\ol \nabla_L D_{KI}{}^J \\
& \quad +D_{KM}{}^J D_{LI}{}^M-D_{LM}{}^J D_{KI}{}^M+\ol T_{KL}{}^M D_{MI}{}^J.
\end{aligned} 
\end{equation}
Hence, the Ricci tensor is given by
\begin{equation}\label{Ricci}
\begin{aligned}
\Ric_{IJ}&=R_J{}^K{}_{KI} \\
&=\ol R_{JI}+\ol\nabla_K D_{IJ}{}^K-\ol\nabla_I D_{KJ}{}^K \\
&\quad +D_{KM}{}^K D_{IJ}{}^M-D_{IM}{}^K D_{KJ}{}^M+\ol T_{KI}{}^MD_{MJ}{}^K \\
&=\ol R_{JI}+\ol\nabla_K D_{IJ}{}^K-\ol\nabla_I D_{KJ}{}^K 
+D_{KM}{}^K D_{IJ}{}^M-D_{MI}{}^K D_{KJ}{}^M.
\end{aligned}
\end{equation}
In the last equality, we have used \eqref{comm-D}. With this formula and Lemma \ref{D}, we can compute 
$E_{IJ}$ modulo negligible terms:

\begin{lem}[{\cite[Lemma 6.5] {Ma1}}, {\cite[Lemma 4.2] {Ma2}}]\label{Einstein-tensor}
Let $\{T, Z_1, Z_{\ol1}\}$ be a unitary admissible frame and $\{\mbox{\boldmath $Z$}_I\}=\{\r\pa_\r, \r^2T, \r Z_1, \r Z_{\overline1}\}$ the associated $\Theta$-frame. Then, modulo negligible terms, the components of the Einstein tensor $E_{IJ}$ are given by
\begin{align*}
&E_{\infty\infty}\equiv -\frac{1}{2}\r\pa_\r(\r\pa_\r-4)\varphi_{00}-\r\pa_\r(\r\pa_\r-2)\varphi_{1\ol1}, \\
&E_{\infty0}\equiv 0, \\
&E_{\infty1}\equiv -\frac{i}{2}(\r\pa_\r+1)\varphi_{01}, \\
&E_{00}\equiv -2\r^4|A|^2-\frac{1}{8}\bigl((\r\pa_\r)^2-6\r\pa_\r-4\bigr)\varphi_{00}+\frac{1}{2}(\r\pa_\r-2)
\varphi_{1\ol1}, \\
&E_{01} \equiv \r^3 A_{11,}{}^1-\frac{1}{8}(\r\pa_\r+1)(\r\pa_\r-5)\varphi_{01}, \\
&E_{1\ol1}\equiv \r^2 \scal^{\rm TW} -\frac{1}{8}\bigl((\r\pa_\r)^2-6\r\pa_\r-8\bigr)\varphi_{1\ol1}
+\frac{1}{8}(\r\pa_\r-4)\varphi_{00}, \\
&E_{11}\equiv i\r^2 A_{11}-\r^4 A_{11,0}-\frac{1}{8}\r\pa_\r (\r\pa_\r-4)\varphi_{11}.
\end{align*}
The components which are not displayed are obtained by the symmetry or by taking the complex conjugates.
\end{lem}
\begin{rem}
We have corrected the value of $E_{00}$ in {\cite[Lemma 6.5] {Ma1}}, {\cite[Lemma 4.2]{Ma2}}, where the term $-2\r^4|A|^2$ is missed, though this modification has no significant effect on the construction of Einstein ACH metric.
\end{rem}

\subsection{The self-dual equation}
Let $\{\bth^I\}$ be the dual $\Theta$-coframe of $\{\mbox{\boldmath $Z$}_I\}$. We take the orientation of 
$\overline X$ such that $\th\wedge d\th\wedge d\r>0$, and define a skew symmetric $\Theta$-tensor
 $\varepsilon_{IJKL}$ by
$$
vol_g=\frac{1}{4!}\varepsilon_{IJKL}\bth^I\wedge\bth^J\wedge\bth^K\wedge\bth^L,
$$
where $vol_g$ is the volume form of $g$. Since $\det (g_{IJ})\equiv -4(1+\varphi_{00}+2\varphi_{1\ol1})$ modulo 
negligible terms, we have 
\begin{align*}
vol_g&=|\det(g_{IJ})|^{1/2}\,i\bth^0\wedge\bth^1\wedge\bth^{\ol1}\wedge\bth^\infty \\
&\equiv (2i+i\varphi_{00}+2i\varphi_{1\ol1})\,\bth^0\wedge\bth^1\wedge\bth^{\ol1}\wedge\bth^\infty,
\end{align*}
and hence
\begin{equation}\label{epsilon}
\varepsilon_{01\ol1\infty}\equiv 2i+i\varphi_{00}+2i\varphi_{1\ol1}.
\end{equation}
Let $P_{IJ}=\frac{1}{2} \Ric_{IJ}-\frac{1}{12} \scal\, g_{IJ}$ be the Schouten tensor, and let 
$$
W_{IJKL}=R_{IJKL}+g_{IK}P_{JL}-g_{JK}P_{IL}+g_{JL}P_{IK}-g_{IL}P_{JK}
$$
be the Weyl curvature. Since $\ol X$ is four dimensional, we can define the anti self-dual part of the Weyl curvature, which is given by
$$
W^-_{IJKL}=\frac{1}{2}\Bigl(W_{IJKL}-\frac{1}{2}\varepsilon_{KL}{}^{PQ}W_{IJPQ}\Bigr).
$$
Note that $W^-_{IJKL}$ has the same symmetry as the Weyl curvature and satisfies 
$$
\frac{1}{2}\varepsilon_{KL}{}^{PQ}W^-_{IJPQ}=-W^-_{IJKL}.
$$
Thus, by \eqref{g-inverse}, \eqref{epsilon}, we have
\begin{align}
\label{W-s1}
&W^-_{\infty0\infty0}\equiv-W^-_{\infty1\infty\ol1}-W^-_{\infty\ol1\infty1}=-2W^-_{\infty1\infty\ol1}, \\
\label{W-s2}
&W^-_{IJ01}=-\varepsilon_{01\ol1\infty}W^-_{IJ}{}^{\,\ol1\infty}\equiv-\frac{i}{2}W^-_{IJ1\infty}, \\ 
\label{W-s3}
&W^-_{IJ1\ol1}=-\varepsilon_{1\ol10\infty}W^-_{IJ}{}^{\,0\infty}\equiv-\frac{i}{2}W^-_{IJ0\infty}
\end{align}
modulo $O(\r)\cdot W^-_{IJKL}$. Since $W^-_{IJKL}=W^-_{KLIJ}$, we also have
\begin{align*}
W^-_{01KL}&\equiv-\frac{i}{2}W^-_{1\infty KL}, \\
W^-_{1\ol1KL}&\equiv-\frac{i}{2}W^-_{0\infty KL}
\end{align*}
modulo $O(\r)\cdot W^-_{IJKL}$. As a consequence, we have the following lemma:
\begin{lem}\label{self-dual-lemma}
Let $m$ be a positive integer. If $W^-_{\infty1\infty1}, W^-_{\infty0\infty1}, W^-_{\infty0\infty0}=O(\r^m)$, then 
$W^-_{IJKL}=O(\r^m)$.
\end{lem}
Thus, in order to solve the self-dual equation $W^-_{IJKL}=O(\r^\infty)$, we only have to deal with the three components indicated above.

Next, we consider the Bianchi identity which relates the self-dual equation to the Einstein equation.
Let $C_{IJK}:=\nabla_K P_{IJ}-\nabla_J P_{IK}$ be the Cotton tensor of $g$ and define the anti self-dual part 
$C^-_{IJK}$ by
$$
C^-_{IJK}=\frac{1}{2}\Bigl(C_{IJK}-\frac{1}{2}\varepsilon_{JK}{}^{PQ}C_{IPQ}\Bigr).
$$
Then, since $\nabla_I\varepsilon_{JKLM}=0$, the Bianchi identity $\nabla^I W_{IJKL}=C_{JKL}$ yields
\begin{equation*}
\nabla^I W^-_{IJKL}=C^-_{JKL}.
\end{equation*}
If $g$ satisfies $E_{IJ}=O(\r^m)$ for some $m\ge 1$, then we have $P_{IJ}=-\frac{1}{4}g_{IJ}+O(\r^m)$ and hence $C^-_{IJK}=O(\r^m)$ since the covariant differentiation does not decrease the vanishing order of a $\Theta$-tensor.   
Therefore, it holds that
$$
E_{IJ}=O(\r^m)\Longrightarrow \nabla^I W^-_{IJKL}=O(\r^m).
$$
To derive the consequence of the latter equation, we will compute 
\begin{equation}\label{div-W-bar}
\begin{aligned}
\nabla^I W^-_{IJKL}&=\ol\nabla ^I W^-_{IJKL}-W^-_{MJKL}D^I{}_I{}^M-W^-_{IMKL}D^I{}_J{}^M \\
& \quad -W^-_{IJML}D^I{}_K{}^M-W^-_{IJKM} D^I{}_L{}^M
\end{aligned}
\end{equation}
modulo $O(\r)\cdot \mathcal{D}W^-_{IJKL}$, where $\mathcal{D}$ is a $\r$-dependent differential operator on $M$. By computations similar to \eqref{W-s1}, \eqref{W-s2}, \eqref{W-s3}, we have
\begin{align}
\label{nabla-W1}
&\nabla^I W^-_{I0\infty0}\equiv -2\nabla^I W^-_{I1\infty\ol1}, \\
\label{nabla-W2}
&\nabla^I W^-_{IJ01}\equiv -\frac{i}{2}\nabla^I W^-_{IJ1\infty}, \\
\label{nabla-W3}
&\nabla^I W^-_{IJ1\ol1}\equiv -\frac{i}{2}\nabla^I W^-_{IJ0\infty}
\end{align}
modulo $O(\r)\cdot \mathcal{D}W^-_{IJKL}$. By \eqref{nabla-W2} and \eqref{nabla-W3}, it suffices to consider the cases where $K=\infty$. Then, taking complex conjugates we may assume that $L=0, 1$, and the case $(J, K, L)=(\ol1, \infty, 0)$ is reduced to the case $(J, K, L)=(1, \infty, 0)$. Moreover, by \eqref{nabla-W1} the case $(J, K, L)=(\ol1, \infty, 1)$ is reduced to the case $(J, K, L)=(0, \infty, 0)$. Thus, it suffices to compute \eqref{div-W-bar} for 
$$
(J, K, L)=(1, \infty, 1), (0, \infty, 0), (0, \infty, 1), (1, \infty, 0), (\infty, \infty, 1), 
(\infty, \infty, 0).
$$ 
By \eqref{nabla-bar}, we have 
\begin{align*}
\ol\nabla ^I W^-_{IJKL}&=\ol\nabla ^\infty W^-_{\infty JKL}+\ol\nabla ^1 W^-_{1JKL}+\ol\nabla ^{\ol1} W^-_{\ol1JKL}+\ol\nabla ^0 W^-_{0JKL} \\
&\equiv\frac{1}{4}\bigl(\r\pa_\r-\#(\infty JKL)\bigr)W^-_{\infty JKL}.
\end{align*}
The other terms in the right-hand side of \eqref{div-W-bar} can be computed by Lemma \ref{D}. 
The final results are:
\begin{equation}\label{div-W}
\begin{aligned}
&\nabla^I W^-_{I1\infty1}\equiv \frac{1}{4}(\r\pa_\r-4)W^-_{\infty1\infty1}, &
&\nabla^I W^-_{I0\infty0}\equiv \frac{1}{4}(\r\pa_\r-6)W^-_{\infty0\infty0}, \\
&\nabla^I W^-_{I0\infty1}\equiv \frac{1}{4}(\r\pa_\r-6)W^-_{\infty0\infty1}, & 
&\nabla^I W^-_{I1\infty0}\equiv \frac{1}{4}(\r\pa_\r-5)W^-_{\infty1\infty0}, \\
&\nabla^I W^-_{I\infty\infty1}\equiv \frac{i}{2} W^-_{\infty0\infty1}, & 
&\nabla^I W^-_{I\infty\infty0}\equiv 0.
\end{aligned}
\end{equation}

Consequently, by an inductive argument, we have the following implication:
$$
E_{IJ}=O(\r^4) \Longrightarrow  W^-_{IJKL}=O(\r^4).
$$
Moreover, if $E_{IJ}=O(\r^5)$ then $W^-_{\infty0\infty0}, W^-_{\infty0\infty1}=O(\r^5)$, but we cannot conclude that $W^-_{\infty1\infty1}=O(\r^5)$. 
Thus, we may use the equation $W^-_{\infty1\infty1}=O(\r^5)$ as a normalization on the metric which is independent of the Einstein equation. We will also use a normalization on the $\r^6$-term in $W^-_{\infty0\infty0}$ whose vanishing is not imposed by the Einstein equation.
To make sure that such normalizations in fact work, we must calculate the variations of $W^-_{\infty1\infty1}$ and $W^-_{\infty0\infty0}$ under the perturbation \eqref{perturb}. 

First, we calculate the relevant components of the curvature tensor modulo negligible terms.
Since the curvature tensor is given by \eqref{curvature} and $R_{IJ\infty K}=-4R_K{}^\infty{}_{IJ}$, we obtain the following result by a straightforward computation using \eqref{nabla-bar} and Lemma \ref{D}:
\begin{equation}\label{curvature-comp}
\begin{aligned}
&R_{\infty0\infty0}\equiv4+\frac{1}{2}\bigl((\r\pa_\r)^2-4\r\pa_\r+8\bigr)\varphi_{00,} &
&R_{01\infty0}\equiv\frac{i}{4}(\r\pa_\r+1)\varphi_{01}, \\
&R_{\infty1\infty1}\equiv\frac{1}{2}\bigl((\r\pa_\r)^2-2\r\pa_\r+2\bigr)\varphi_{11}, &
&R_{01\infty1}\equiv\r^2A_{11}-\frac{i}{4}\r\pa_\r \varphi_{11}, \\
&R_{0\ol1\infty1}\equiv -\frac{i}{2}+\frac{i}{4}(\r\pa_\r-2)\varphi_{00}-\frac{i}{4}\r\pa_\r\varphi_{1\ol1}, &
&R_{1\ol1\infty1}\equiv \frac{3i}{4}(\r\pa_\r-1)\varphi_{01}, \\
&R_{1\ol1\infty0}\equiv -i+\frac{i}{2}(\r\pa_\r-2)\varphi_{00}-\frac{i}{2}\r\pa_\r\varphi_{1\ol1}.
\end{aligned}
\end{equation}
These equations enable us to compute the variations of the curvature components under the perturbation 
\eqref{perturb}, which we denote by putting `$\d$' to each component. For example, by the first equation in 
\eqref{curvature-comp}, we have 
$$
\d R_{\infty0\infty0}=\frac{1}{2}(m^2-4m+8)\psi_{00}+O(\r^{m+1}).
$$
Next, we calculate the variation of the Schouten tensor 
$$
P_{IJ}=\frac{1}{2}E_{IJ}-\frac{1}{12}(E_{K}{}^K+3)g_{IJ}.
$$
Since $E_{IJ}=O(\r)$ by Lemma \ref{Einstein-tensor}, we have  
$$
\d P_{IJ}=\frac{1}{2}\d E_{IJ}-\frac{1}{12}g^{KL}(\d E_{KL}) g_{IJ}-\frac{1}{4}\d g_{IJ}+O(\r^{m+1}),
$$
which yields
\begin{equation}\label{d-P}
\begin{aligned}
&\d P_{\infty\infty}=-\frac{1}{6}(m^2-3m-1)\psi_{00}-\frac{1}{6}(2m^2-m+2)\psi_{1\ol1}+O(\r^{m+1}), \\
&\d P_{\infty0}=O(\r^{m+1}), \\
&\d P_{\infty1}=-\frac{i}{4}(m+1)\psi_{01}+O(\r^{m+1}), \\
&\d P_{00}=-\frac{1}{24}(m^2-6m-1)\psi_{00}+\frac{1}{24}(m^2+m-14)\psi_{1\ol1}+O(\r^{m+1}), \\
&\d P_{01}=-\frac{1}{16}(m^2-4m-1)\psi_{01}+O(\r^{m+1}), \\
&\d P_{1\ol1}=\frac{1}{48}(m^2-3m-10)\psi_{00}-\frac{1}{48}(m^2-8m-8)\psi_{1\ol1}+O(\r^{m+1}), \\
&\d P_{11}=-\frac{1}{16}(m^2-4m+4)\psi_{11}+O(\r^{m+1}).
\end{aligned}
\end{equation}
From \eqref{g-inverse}, \eqref{curvature-comp}, and \eqref{d-P}, we have 

\begin{equation}\label{d-W}
\begin{aligned}
&W_{\infty1}{}^{0\ol1}=O(\r), \quad W_{\infty0}{}^{\ol11}=-i+O(\r), \\
&\d W_{\infty1}{}^{0\ol1}=-\frac{i}{4}(m-2)\psi_{11}+O(\r^{m+1}), \\
&\d W_{\infty0}{}^{\ol11}=\frac{i}{2}(m-2)\psi_{00}-\frac{i}{2}(m-4)\psi_{1\ol1}+O(\r^{m+1}), \\
&\d W_{\infty1\infty1}=\frac{1}{4}(m^2-4)\psi_{11}+O(\r^{m+1}), \\
&\d W_{\infty0\infty0}=\frac{1}{6}(m^2-3m+20)\psi_{00}-\frac{1}{6}(m^2-2m+16)\psi_{1\ol1}+O(\r^{m+1}).
\end{aligned}
\end{equation}
Finally, by \eqref{epsilon} and \eqref{d-W}, we obtain

\begin{align}
\d W^-_{\infty1\infty1}&=\frac{1}{2}(\d W_{\infty1\infty1}-
\d\varepsilon_{\infty10\ol1}\cdot W_{\infty1}{}^{0\ol1}-\varepsilon_{\infty10\ol1}\cdot\d W_{\infty1}{}^{0\ol1})         +O(\r^{m+1}) \label{var-W1} \\
&=\frac{1}{8}(m^2-2m)\psi_{11}+O(\r^{m+1}), \notag \\
\d W^-_{\infty0\infty0}&=\frac{1}{2}(\d W_{\infty0\infty0}-
\d\varepsilon_{\infty0\ol11}\cdot W_{\infty0}{}^{\ol11}-\varepsilon_{\infty0\ol11}\cdot\d W_{\infty0}{}^{\ol11})  +O(\r^{m+1}) \label{var-W2} \\
&=\frac{1}{12}(m^2+3m+2)\psi_{00}-\frac{1}{12}(m^2+4m+4)\psi_{1\ol1}+O(\r^{m+1}). \notag
\end{align}

\subsection{Bianchi identities}
Since the Einstein equation is an overdetermined system, we need some relations which are satisfied by the components of the Einstein tensor in order to construct a formal solution to the Einstein equation. Some of them are given by the Bianchi identity $g^{IJ}\nabla_K E_{IJ}=2g^{IJ}\nabla_I E_{JK}$:
\begin{lem}[{\cite[Lemma 6.6]{Ma1}}, {\cite[Lemma 6.1]{Ma2}}]
Suppose $g$ satisfies $E_{IJ}=O(\r^m)$ for an integer $m\ge1$. Then, we have
\begin{align}
(m-8)E_{\infty\infty}-4(m-4)E_{00}-8(m-2)E_{1\ol1}&=O(\r^{m+1}), \label{Bianchi1} \\
(m-6)E_{\infty0}&=O(\r^{m+1}), \label{Bianchi2} \\
(m-5)E_{\infty1}-4iE_{01}&=O(\r^{m+1}). \label{Bianchi3}
\end{align}
\end{lem}
We will also use some equations obtained from the Bianchi identity $\nabla^I W^-_{IJKL}=C^-_{JKL}$ in the construction of $g$. Since the Cotton tensor is given by 
$$
C_{IJK}=\frac{1}{2}(\nabla_K E_{IJ}-\nabla_J E_{IK})-\frac{1}{12}\bigl((\nabla_K E_L{}^L) g_{IJ}-(\nabla_J E_L{}^L) g_{IK}\bigr),
$$
we can compute the components $C^-_{IJK}$ in terms of $E_{IJ}$ by using \eqref{nabla-bar}, \eqref{g-inverse}, \eqref{epsilon}, and Lemma \ref{D}. As a result, we have the following lemma:
\begin{lem}
Suppose $g$ satisfies $E_{IJ}=O(\r^m)$ for an integer $m\ge1$. Then, we have
\begin{align}
&C^-_{1\infty1}=-\frac{1}{4}(m-2)E_{11}+O(\r^{m+1}), \label{C1} \\
&C^-_{0\infty0}=-\frac{5}{24}m E_{00}+\frac{1}{96}(m-12)E_{\infty\infty}+\frac{1}{12}(m+6)E_{1\ol1}
+O(\r^{m+1}), \label{C2}  \\
&C^-_{\infty\infty0}=-\frac{1}{4}(m-2)E_{\infty0}+O(\r^{m+1}). \label{C3}
\end{align}
\end{lem}

\section{Construction of the metric}

\subsection{Formal solution to the self-dual Einstein equation}
Let $M$ be a three dimensional strictly pseudoconvex CR manifold. We fix a contact form $\th$ and construct a one parameter family of ACH metrics $g^\lambda$ on $\ol X=M\times [0, \infty)_\r$ which are in normal form with respect to $\th$ and satisfy the Einstein equation to infinite order.
The parameter $\lambda\in\mathbb{R}$ is involved in the normalization on the $\r^6$-term in $g^\lambda$, 
and if $\lambda=0$ the metric is self-dual to infinite order. As in the previous section, we take the $\Theta$-frame $\{\mbox{\boldmath $Z$}_I\}=\{\r\pa_\r, \r^2T, \r Z_1, \r Z_{\overline1}\}$ associated with a unitary admissible frame $\{T, Z_1, Z_{\ol1}\}$ on $M$. We suppress the superscript $\lambda$ in the following.

First we show a lemma which assures that our normalization condition is independent of the choice of $\th$.
\begin{lem}\label{eta}
Suppose that an ACH metric $g$ on $\ol X$ satisfies $W^-_{IJKL}=O(\r^6)$, and let $\r_\th$ be the model defining function associated with a contact form $\th$. Then, 
\begin{equation}\label{def-eta}
\eta_\th:=\bigl(\r_\th^{-6}W^-_{\infty0\infty0}\bigr)\big|_M
\end{equation}
satisfies $\eta_{\wh\th}=e^{-3\U}\eta_\th$ for the rescaling $\wh\th=e^{\U}\th$.
\end{lem}
\begin{proof}
By Proposition \ref{Z-M}, $\mbox{\boldmath $Z$}_\infty|_M$ and $\mbox{\boldmath $Z$}_0|_M$ are determined by $g$ and independent of $\th$. Thus, we have ${\wh W}^-_{\infty0\infty0}=W^-_{\infty0\infty0}+O(\r^7)$. Since $\r_{\wh\th}=e^{\U/2}\r_\th+O(\r^2)$ by \eqref{rho}, we obtain $\eta_{\wh\th}=e^{-3\U}\eta_\th$.
\end{proof}

This lemma implies that if $\eta_\th$ has a universal expression in terms of the Tanaka--Webster connection, then it defines a CR invariant $\eta\in\calE(-3, -3)$. Since such a CR invariant is necessarily 
a multiple of the obstruction density \cite{G}, we are led to the CR invariant normalization 
$\eta=\lambda \mathcal{O}$.

Now we construct the metric and prove Theorem \ref{main-thm}. We start with an arbitrary normal form ACH metric $g^{\tiny{(1)}}_{IJ}$, which automatically satisfies $E_{IJ}=O(\r)$ by Lemma \ref{Einstein-tensor}.  
Supposing that we have a normal form ACH metric $g^{(m)}_{IJ}$ such that $E_{IJ}=O(\r^m)$, we consider a perturbed metric 
$$
g^{(m+1)}_{IJ}=g^{(m)}_{IJ}+\psi_{IJ}, \quad \psi_{\infty J}=0, \ \psi_{IJ}=O(\r^m)
$$
and try to solve $E_{IJ}=O(\r^{m+1})$. We also take $W^-_{IJKL}$ into consideration in each inductive step 
by using the following equations modulo $O(\r)\cdot\mathcal{D}W^-_{IJKL}$ from \eqref{div-W}:
\begin{align}
&\nabla^I W^-_{I1\infty1}\equiv \frac{1}{4}(\r\pa_\r-4)W^-_{\infty1\infty1},  \label{W1} \\
&\nabla^I W^-_{I0\infty0}\equiv \frac{1}{4}(\r\pa_\r-6)W^-_{\infty0\infty0}, \label{W2} \\
&\nabla^I W^-_{I\infty\infty1}\equiv \frac{i}{2} W^-_{\infty0\infty1}, \label{W3} \\
&\nabla^I W^-_{I\infty\infty0}\equiv 0. \label{W4}
\end{align}

By Lemma \ref{Einstein-tensor}, the variation of $E_{IJ}$ is given by 
\begin{align}
&\d E_{\infty\infty}=-\frac{1}{2}m(m-4)\psi_{00}-m(m-2)\psi_{1\ol1}+O(\r^{m+1}),  \label{dE1} \\
&\d E_{\infty0}=O(\r^{m+1}), \label{dE2} \\
&\d E_{\infty1}=-\frac{i}{2}(m+1)\psi_{01}+O(\r^{m+1}), \label{dE3} \\
&\d E_{00}=-\frac{1}{8}(m^2-6m-4)\psi_{00}+\frac{1}{2}(m-2)\psi_{1\ol1}+O(\r^{m+1}), \label{dE4} \\
&\d E_{01}=-\frac{1}{8}(m+1)(m-5)\psi_{01}+O(\r^{m+1}), \label{dE5} \\
&\d E_{1\ol1}=-\frac{1}{8}(m^2-6m-8)\psi_{1\ol1}+\frac{1}{8}(m-4)\psi_{00}+O(\r^{m+1}), \label{dE6} \\
&\d E_{11}=-\frac{1}{8}m(m-4)\psi_{11}+O(\r^{m+1}). \label{dE7}
\end{align}
The determinant of the coefficients of \eqref{dE4} and \eqref{dE6} as a system of linear equations for 
$\psi_{00}$ and $\psi_{1\ol1}$ is 
$$
\det
\begin{pmatrix}
-\frac{1}{8}(m^2-6m-4) & \frac{1}{2}(m-2) \\
\frac{1}{8}(m-4) & -\frac{1}{8}(m^2-6m-8)
\end{pmatrix}
=\frac{1}{64}m(m+2)(m-6)(m-8).
$$

First we consider the case of $m\le5$, where the determinant is nonzero. 
We determine $\psi_{00}$, $\psi_{1\ol1}$ (modulo $O(\r^{m+1})$) by \eqref{dE4}, \eqref{dE6} so that $E_{00}, E_{1\ol1}=O(\r^{m+1})$ holds. Then by the Bianchi identities \eqref{Bianchi1}, \eqref{Bianchi2}, we have 
$E_{\infty\infty}, E_{\infty0}=O(\r^{m+1})$. We determine $\psi_{01}$ by \eqref{dE3} to obtain $E_{\infty1}=O(\r^{m+1})$. Then, \eqref{Bianchi3} gives $E_{01}=O(\r^{m+1})$. When $m\le3$, \eqref{dE7} determines $\psi_{11}$ so that $E_{11}=O(\r^{m+1})$, thus we have $E_{IJ}=O(\r^{m+1})$. Moreover, by \eqref{W1}--\eqref{W3} and Lemma \ref{self-dual-lemma}, we also have $W^-_{IJKL}=O(\r^{m+1})$. When $m=4$, we cannot use \eqref{dE7} to obtain $E_{11}=O(\r^5)$. However, since $W^-_{IJKL}=O(\r^4)$, it follows from \eqref{W1} that
\begin{align*}
C^-_{1\infty1}=\nabla^I W^-_{I1\infty1}=\frac{1}{4}(4-4)W^-_{\infty1\infty1}+O(\r^5) 
=O(\r^5),
\end{align*}
so we have $E_{11}=O(\r^5)$ by \eqref{C1}. (This also follows from the fact that the CR obstruction tensor $\mathcal{O}_{11}=(\r^{-4}E_{11})|_M$ vanishes in three dimension; see \cite{Ma1, Ma2}.) Thus, we have $E_{IJ}=O(\r^5)$ and by \eqref{W1}--\eqref{W3}, it holds that 
$$
W^-_{\infty1\infty1}=O(\r^4), \quad W^-_{\infty0\infty0}, W^-_{\infty0\infty1}=O(\r^5).
$$
We can choose $\psi_{11}$ so that $W^-_{\infty1\infty1}=O(\r^5)$ holds since 
$$
\d W^-_{\infty1\infty1}=\psi_{11}+O(\r^5)
$$
by \eqref{var-W1}. Thus we obtain $g^{(5)}_{IJ}$ unique modulo $O(\r^{5})$ with $E_{IJ}, W^-_{IJKL}=O(\r^5)$. 
When $m=5$, we can construct $g^{(6)}_{IJ}$ with $E_{IJ}=O(\r^6)$ in the same way for  $m\le3$ and 
we also have $W^-_{IJKL}=O(\r^6)$ by \eqref{W1}--\eqref{W3}. 
 
Next we consider the case of $m=6$, where the equations \eqref{dE1}, \eqref{dE4}, \eqref{dE6} are not pairwise independent. We determine $\psi_{01}$ by \eqref{dE3} so that $E_{\infty1}=O(\r^7)$. Then we also have 
$E_{01}=O(\r^7)$ by \eqref{Bianchi3}. We determine $\psi_{11}$ by \eqref{dE7} and obtain $E_{11}=O(\r^7)$.
By \eqref{var-W2}, we have 
\begin{equation}\label{dW}
\d W^-_{\infty0\infty0}=\frac{14}{3}\psi_{00}-\frac{16}{3}\psi_{1\ol1}+O(\r^7).
\end{equation}
We use this equation and \eqref{dE1} to determine $\psi_{00}, \psi_{1\ol1}$ so that 
$$
E_{\infty\infty}=O(\r^7), \quad \eta=\lambda\mathcal{O}
$$
holds. Thus we have determined $g^{(7)}_{IJ}$ and we must check that it also satisfies 
$E_{00}, E_{1\ol1}, E_{\infty0}=O(\r^7)$. Since $W^-_{IJKL}=O(\r^6)$, by \eqref{W2} we have
$$
C^-_{0\infty0}=\frac{1}{4}(6-6)W^-_{\infty0\infty0}+O(\r^7)=O(\r^7).
$$
Then it follows from \eqref{C2} that 
$$
-\frac{5}{4}E_{00}+E_{1\ol1}=O(\r^7).
$$
Also, \eqref{Bianchi1} gives
$$
E_{00}+4E_{1\ol1}=O(\r^7).
$$
Therefore, we have $E_{00}, E_{1\ol1}=O(\r^7)$. Moreover, by $W^-_{IJKL}=O(\r^6)$ and the equation \eqref{W4}, it holds that $C^-_{\infty\infty0}=O(\r^7)$, which implies $E_{\infty0}=O(\r^7)$ by \eqref{C3}. 
Thus, $g^{(7)}_{IJ}$ satisfies $E_{IJ}=O(\r^7)$, $W^-_{IJKL}=O(\r^6)$, and $\eta=\lambda\mathcal{O}$. We note that it satisfies $W^-_{IJKL}=O(\r^7)$ when $\lambda=0$.

When $m=7$, we can determine $g^{(8)}_{IJ}$ so that it satisfies $E_{IJ}=O(\r^8)$ in the same way as for $m\le3$. If $\lambda=0$, it also satisfies $W^-_{IJKL}=O(\r^8)$ by \eqref{W1}--\eqref{W3}.

Let us consider the case of $m=8$. In this case, the equations \eqref{dE4} and \eqref{dE6} are not independent. We use \eqref{dE1}, \eqref{dE4} to determine $\psi_{00}, \psi_{1\ol1}$ so that $E_{\infty\infty}, E_{00}=O(\r^9)$. Then \eqref{Bianchi1} gives $E_{1\ol1}=O(\r^9)$. We determine $\psi_{01}$ and $\psi_{11}$ by \eqref{dE3} and \eqref{dE7} respectively and obtain $E_{\infty1}, E_{11}=O(\r^9)$. By \eqref{Bianchi2}, \eqref{Bianchi3}, we have $E_{\infty0}, E_{01}=O(\r^9)$. Thus we have constructed $g^{(9)}_{IJ}$ with 
$E_{IJ}=O(\r^9)$, which satisfies $W^-_{IJKL}=O(\r^9)$ when $\lambda=0$ by \eqref{W1}--\eqref{W3}.

Finally, let $m\ge9$. In this case, the equation $E_{IJ}=O(\r^{m+1})$ determines $g^{(m+1)}_{IJ}$ in the same way as in $m\le3$, and it satisfies $W^-_{IJKL}=O(\r^{m+1})$ by \eqref{W1}--\eqref{W3} when $\lambda=0$.

Consequently, we can construct all $g^{(m+1)}_{IJ}$ inductively,  and by Borel's lemma we obtain a solution $g^\lambda_{IJ}$ to 
$$
E_{IJ}=O(\r^{\infty}), \quad W^-_{IJKL}=O(\r^6), \quad \eta=\lambda\mathcal{O},
$$ 
which is unique modulo $O(\r^\infty)$. By the construction, $g^0_{IJ}$ satisfies $W^-_{IJKL}=O(\r^\infty)$.
Thus we complete the proof of Theorem \ref{main-thm}.

\subsection{Dependence on $\lambda$}
We can read off the dependence of $g^\lambda_{IJ}$ on the parameter $\lambda$ from the construction.

\begin{prop}\label{exp}
The metric $g^\lambda_{IJ}$ admits the following asymptotic expansion:
$$
g^\lambda_{IJ}\sim g^0_{IJ}+\sum_{k=1}^\infty \lambda^k \r^{6k}\phi^{(k)}_{IJ}(\r), \quad
\phi^{(k)}_{I\infty}=\phi^{(k)}_{01}=\phi^{(k)}_{11}=0.
$$
Here, $\phi^{(k)}_{IJ}(\rho)$ is a formal power series in $\rho$.
\end{prop}
\begin{proof}
We write the Taylor expansion of $g^\lambda_{IJ}-g^0_{IJ}$ as 
$$
g^\lambda_{IJ}-g^0_{IJ}\sim \sum_{k=0}^\infty \r^k \Phi^{\lambda, k}_{IJ}, \quad \Phi^{\lambda, k}_{I\infty}=0.
$$
Then, it suffices to show that $\Phi^{\lambda, k}_{01}=\Phi^{\lambda, k}_{11}=0$ and each $\Phi^{\lambda, k}_{IJ}$ is a polynomial in $\lambda$ of degree $\le k/6$. First, we note that 
$g^\lambda_{IJ}-g^0_{IJ}=O(\r^6)$ hence $\Phi^{\lambda, k}_{IJ}=0$ for $k\le5$.
Since both $g^\lambda_{IJ}$ and $g^0_{IJ}$ satisfy $E_{\infty1}, E_{11}=O(\r^\infty)$, we have $\Phi^{\lambda, k}_{01}=\Phi^{\lambda, k}_{11}=0$ also for $k\ge6$ by \eqref{dE3}, \eqref{dE7}. From \eqref{dE1}, \eqref{dW}, we see that 
$\Phi^{\lambda, 6}_{00}$ and $\Phi^{\lambda, 6}_{1\ol1}$ are determined by 
\begin{align*}
-6\Phi^{\lambda, 6}_{00}-24\Phi^{\lambda, 6}_{1\ol1}&=0, \\
\frac{14}{3}\Phi^{\lambda, 6}_{00}-\frac{16}{3}\Phi^{\lambda, 6}_{1\ol1}&=\lambda\mathcal{O}.
\end{align*}
Thus we have ${\rm deg}\,\Phi^{\lambda, 6}_{00}={\rm deg}\,\Phi^{\lambda, 6}_{1\ol1}=1$. Now we shall prove 
${\rm deg}\, \Phi^{\lambda, k}_{00}$, ${\rm deg}\, \Phi^{\lambda, k}_{1\ol1}\le k/6$ by the induction on $k$.
When $k\ge7$, $\Phi^{\lambda, k}_{00}$ and $\Phi^{\lambda, k}_{1\ol1}$ are determined by the condition 
$\pa^k_\r E_{00}|_{\r=0}=\pa^k_\r E_{1\ol1}|_{\r=0}=0$ for $k\neq8$ and 
$\pa^k_\r E_{\infty\infty}|_{\r=0}=\pa^k_\r E_{00}|_{\r=0}=0$ for $k=8$. These conditions can be regarded as a system of linear equations for $\Phi^{\lambda, k}_{00}$ and $\Phi^{\lambda, k}_{1\ol1}$, and in view of 
\eqref{nabla-bar}, \eqref{D-formula}, \eqref{Ricci},
the terms involving the other components are linear combinations of 
$$
\mathcal{D}_1\Phi^{\lambda, l_1}_{I_1J_1}\cdots\mathcal{D}_p\Phi^{\lambda, l_p}_{I_p J_p} \quad 
(l_1+\cdots+ l_p\le k,\ l_j<k),
$$
where $\mathcal{D}_j$ is a differential operator on $M$. Then, by the induction hypothesis, we have 
$$
{\rm deg}\, \Phi^{\lambda, k}_{IJ}\le \frac{l_1+\cdots+ l_p}{6}\le \frac{k}{6}.
$$
Thus, we complete the proof.
\end{proof}
\subsection{Evenness}
 Let $g$ be a normal form ACH metric on $M\times[0, \infty)_\r$. Then it can be written in the form 
\begin{equation}\label{normal-g}
g=\frac{h_\r+4d\r^2}{\r^2},
\end{equation}
where $h_\r$ is a family of Riemannian metrics on $M$. We say $g$ is {\it even} when $h_\r$ has even Laurent expansion at $\r=0$. In other words, $g$ is even if and only if the components 
$g_{00}, g_{11}, g_{1\ol1}$ are even in $\r$, and $g_{01}$ is odd in $\r$. An ACH metric is said to be even if its normal form is even for any choice of $\th$.
\begin{prop}
The ACH metric $g^\lambda$ is even.
\end{prop}
\begin{proof}
Fix a contact form $\th$ and suppose $g^\lambda$ is in the normal form as \eqref{normal-g}. By using the Laurent expansion of $h_\r$, we can regard the right-hand side of \eqref{normal-g} as an ACH metric $g^\lambda_{-}$ defined on $M\times (-\infty, 0]_\r$. Then, $g^\lambda_-$ also satisfies
$$
E_{IJ}=O(\r^\infty), \quad 
W^-_{IJKL}=O(\r^6), \quad \eta=\lambda\mathcal{O}
$$
with respect to the orientation satisfying
$$
i\bth^0\wedge\bth^1\wedge\bth^{\ol1}\wedge\bth^\infty
=i\r^{-5}\th\wedge\th^1\wedge\th^{\ol1}\wedge d\r>0.
$$
We consider the ACH metric $\iota^* g^\lambda_-$ on $M\times[0, \infty)_\r$, where 
$\iota(x, \r):=(x, -\r)$. Since $\iota$ preserves the orientation, $\iota^* g^\lambda_-$ satisfies
$$
E_{IJ}=O(\r^\infty), \quad W^-_{IJKL}=O(\r^6).
$$
Noting that $\iota_* \mbox{\boldmath $Z$}_\infty=\mbox{\boldmath $Z$}_\infty$ and 
$\iota_* \mbox{\boldmath $Z$}_0=\mbox{\boldmath $Z$}_0$, we have
\begin{align*}
\r^{-6}W^-[\iota^*g^\lambda_-]_{\infty0\infty0}
&=(\iota^*\r)^{-6}(\iota^*W^-[g^\lambda_-])
(\mbox{\boldmath $Z$}_\infty, \mbox{\boldmath $Z$}_0, \mbox{\boldmath $Z$}_\infty, \mbox{\boldmath $Z$}_0) \\
&=\iota^* \bigl(\r^{-6}W^-[g^\lambda_-]
(\iota_*\mbox{\boldmath $Z$}_\infty, \iota_*\mbox{\boldmath $Z$}_0, \iota_*\mbox{\boldmath $Z$}_\infty, \iota_*\mbox{\boldmath $Z$}_0)\bigr) \\
&=\iota^*(\r^{-6}W^-[g^\lambda_-]_{\infty0\infty0}).
\end{align*}
Thus, $\iota^*g^\lambda_-$ also satisfies $\eta=\lambda\mathcal{O}$. Therefore, by the uniqueness we 
obtain $\iota^*g^\lambda_- = g^\lambda+O(\r^\infty)$, which implies that $g^\lambda$ is even.
\end{proof}

\section{CR GJMS operators}
Matsumoto \cite{Ma3} generalized the CR GJMS operators to partially integrable CR manifolds via Dirichlet-to-Neumann type operators associated with eigenvalue equations for the Laplacian of the ACH metric. 
In dimension three, it is stated as follows:

\begin{thm}[{\cite[Theorem 3.3]{Ma3}}]
Let $M$ be a three dimensional strictly pseudoconvex CR manifold and $g$ an ACH metric on a $\Theta$-manifold $\ol X$ with the boundary $M$. Let $\theta$ be a contact form on $M$ and let $\r$ be the model defining function associated with $\th$. Then, for any $k\in\mathbb{N}_+$ and $f\in C^\infty(M)$, there exist $F, G\in C^\infty(\ol X)$ with $F|_M=f$ such that the function 
$u:=\r^{-k+2}F+(\r^{k+2}\log\r) G$ satisfies
$$
\Bigl(\Delta+\frac{k^2}{4}-1\Bigr)u=O(\r^\infty),
$$
where $\Delta=-g^{IJ}\nabla_I\nabla_J$ is the Laplacian of $g$. The function $G$ is unique modulo $O(\r^\infty)$ and $P_{2k}f:=(-1)^{k+1}k!(k-1)!/2\cdot G|_M$ defines a formally self-adjoint linear differential operator $\calE(k/2-1, k/2-1)\rightarrow \calE(-k/2-1, -k/2-1)$ which is independent of the choice of $\th$ and has the principal part $\Delta_b^k$.
\end{thm}

We apply this theorem to our metric $g^\lambda$. Since $g^\lambda$ is determined to infinite order and 
the Taylor expansion has a universal expression in terms of the pseudo-hermitian structure, 
the operator $P^\lambda_{2k}$ has a universal expression in terms of Tanaka--Webster connection. Thus, we obtain the CR GJMS operators $P^\lambda_{2k}$ for all $k\ge1$. 

In order to prove that $P^\lambda_{2k}$ is a polynomial in $\lambda$ of degree $\le k/3$, we will review the detail of its construction. A linear differential operator on $\ol X$ is called a $\Theta$-{\it differential operator} if it is the sum of linear differential operators of the form 
$a Y_1\cdots Y_N$, where $a\in C^\infty(\ol X)$ and $Y_j\in\Gamma(^\Theta T\ol X)$. Note that a $\Theta$-differential operator preserves the subspace $\r^m C^\infty(\ol X)\subset C^\infty(\ol X)$ for each $m\ge1$. We fix a contact form $\th$ and denote the associated Tanaka--Webster connection by 
$\nabla^{\rm TW}$. Suppose that $g^\lambda$ is of the normal form
$$
g^\lambda=k_\r+4\frac{d\r^2}{\r^2}
$$
for $\th$, where $k_\r$ is a family of Riemannian metrics on $M$. Then, the Laplacian 
$\Delta$ of $g^\lambda$ is written as
\begin{equation}\label{Delta}
\Delta=-\frac{1}{4}(\r\pa_\r)^2+\r\pa_\r+\r^2\Delta_b -\r^4 T^2+\r\Psi
\end{equation}
with the $\Theta$-differential operator $\Psi$ defined by
\begin{equation*}
\begin{aligned}
\Psi f&=-\frac{1}{8}(\pa_\r\log\det k_\r)\r\pa_\r f-\r^{-1}\bigl((k^{-1}_\r)^{ij}-(k^{-1}_0)^{ij}\bigr)\nabla^{\rm TW}_i
\nabla^{\rm TW}_j f \\
&\quad +\frac{1}{2}(k^{-1}_\r)^{ij}(k^{-1}_\r)^{kl}\r^{-1}\bigl(\nabla^{\rm TW}_i(k_\r)_{jk}+
\nabla^{\rm TW}_j(k_\r)_{ik}-\nabla^{\rm TW}_k(k_\r)_{ij}\bigr)\nabla^{\rm TW}_l f.
\end{aligned}
\end{equation*}
Here, the components are with respect to a $\Theta$-frame $\{\mbox{\boldmath $Z$}_I\}$, and 
we note that $(k^{-1}_\r)^{ij}-(k^{-1}_0)^{ij}$ and $\nabla^{\rm TW}_i(k_\r)_{jk}+
\nabla^{\rm TW}_j(k_\r)_{ik}-\nabla^{\rm TW}_k(k_\r)_{ij}$ are $O(\r)$ by \eqref{Gamma-bar}.
In particular, $\Psi$ involves $\pa_\r g^\lambda_{IJ}$ but not higher order derivatives.

Given a function $f\in C^\infty(M)$, we try to solve the equation 
$$
\Bigl(\Delta+\frac{k^2}{4}-1\Bigr)(\r^{-k+2}F)=0
$$
for $F\in C^\infty(\ol X)$ with $F|_M=f$. Let $F\sim \sum_{j=0}^\infty f^{(j)}\r^{j},\ (f^{(j)}\in C^\infty(M))$ be the Taylor expansion of $F$ along $M$. By \eqref{Delta}, we have 
\begin{equation}\label{inductive}
\Bigl(\Delta+\frac{k^2}{4}-1\Bigr)(\r^{-k+2+j}f^{(j)})=\r^{-k+2+j}
\Bigl(-\frac{1}{4}j(j-2k)f^{(j)}+\r \mathcal{D}_j f^{(j)}\Bigr),
\end{equation}
where $\mathcal{D}_j$ is a $\r$-dependent linear differential operator on $M$. Starting with $f^{(0)}=f$, we  inductively define $f^{(j)}$ so that $F$ satisfies 
$$
\Bigl(\Delta+\frac{k^2}{4}-1\Bigr)(\r^{-k+2}F)=O(\r^{-k+3+j}).
$$
Let $\mathcal{D}_j\sim\sum_{l=0}^\infty \mathcal{D}_j^{(l)}\r^l$ be the Taylor expansion of $\mathcal{D}_j$. Then, by \eqref{inductive} $f^{(j)}$ is determined for $j\le2k-1$ as 
$$
f^{(j)}=\frac{4}{j(j-2k)}\sum_{l=0}^{j-1}\mathcal{D}_l^{(j-1-l)}f^{(l)}.
$$
We cannot define $f^{(2k)}$ due to the vanishing of the coefficient of $f^{(2k)}$ in \eqref{inductive}, and we need to introduce the logarithmic term $(\r^{k+2}\log\r) G$ in which the coefficient $G|_M$ is a multiple of 
$$
\Bigl(\r^{-k-2}\Bigl(\Delta+\frac{k^2}{4}-1\Bigr)(\r^{-k+2}F)\Bigl)\Big|_M.
$$
Therefore, up to a constant multiple, $P^\lambda_{2k}f$ is given by 
$$
\sum_{j=0}^{2k-1}\mathcal{D}_j^{(2k-1-j)}f^{(j)}.
$$
Since $\Psi$ involves only $g^\lambda_{IJ}$ and their first order derivatives in $\r$, $\mathcal{D}_j^{(l)}$ involves $\pa_\r^m g^\lambda_{IJ}$ for $m\le l+1$. Consequently, $P^\lambda_{2k}$ is written in terms of 
$\pa_\r^m g^\lambda_{IJ}\ (m\le 2k)$, and by Proposition \ref{exp} it is a polynomial in $\lambda$ of degree
$\le k/3$. Thus we complete the proof of Theorem \ref{GJMS}.

\section{Convergence of the formal solutions}

We will prove Theorem \ref{convergence}, which asserts that the formal solution $g^\lambda$ converges to a real analytic ACH metric near $M$ when $M$ is a real analytic CR manifold. In the case of $\lambda=0$, this recovers the result of Biquard \cite{B}.
The key tool is the result of  Baouendi--Goulaouic \cite{BG} on the unique existence of the solution to a singular nonlinear Cauchy problem. Let us state their theorem in a form which fits to our setting. 

We regard local coordinates $(x, \r)$ of $M\times[0, \infty)_\r$ as complex 
variables and consider an equation for a $\mathbb{C}^N$-valued holomorphic function $v(x, \r)$ of the form
\begin{equation}\label{cauchy}
\begin{aligned}
(\r\pa_\r)^m v+ A_{m-1}(\r\pa_\r)^{m-1}v&+\cdots+A_0 v \\
 &=F(x, \r, \{(\r\pa_\r)^l\pa_x^\a(\r v)\}_{l+|\a|\le m,\ l<m}),
\end{aligned}
\end{equation}
where $A_j$ is an $N\times N$ matrix and $F(x, \r, \{y_{l, \a}\}_{l+|\a|\le m,\ l<m})$ is a holomorphic function near 0. For each $k\in\mathbb{N}$, we set
$$
\mathcal{P}(k):=k^m I+k^{m-1}A_{m-1}+\cdots+A_0,
$$
where $I$ is the identity matrix of size $N$. Then, by \cite[Theorem 3.1]{BG} we have the following theorem:
\begin{thm}\label{solution}
If $\det\mathcal{P}(k)\neq0$ for all $k\in\mathbb{N}$, the equation \eqref{cauchy} has a unique holomorphic solution $v(x, \r)$ near $(0, 0)$. 
\end{thm}

In the original statement of \cite[Theorem 3.1]{BG}, the right-hand side of the equation \eqref{cauchy} is replaced by $G(\r, \{(\r\pa_\r)^l\pa_x^\a(\r v)\}_{l+|\a|\le m,\ l<m})$
with $G$ a $C^\infty$-map
$$
G: \mathbb{C}\times B^{N^\prime}\longrightarrow B,
$$
where $B$ is the Banach space of $\mathbb{C}^N$-valued bounded holomorphic functions of
$x$ on a fixed polydisc, and $N^\prime$ is the number of multiindices $(l, \a)$ such that 
$l+|\a|\le m,\ l<m$. Also, the solution $v$ is given as a $C^\infty$-function of $\r$ valued in $B$. In our equation \eqref{cauchy}, $G$ is given by $G(\r, \{y_{l, \a}\}):=F(x, \r, \{y_{l, \a}(x)\})$. Since  
this is analytic in $\r$, it follows from \cite[Remark 2.2]{BG} and the proof of \cite[Theorem 3.1]{BG} that the solution $v(x, \r)$ is $C^\infty$ and $v(x, \r^m)$ is holomorphic, which implies that $v(x, \r)$ itself is holomorphic. Thus we obtain Theorem \ref{solution} as a special case of their theorem.

Now we apply this theorem to our case. We assume that $M$ is a real analytic CR manifold. Let $g^\lambda_{IJ}$ be the components of the formal solution $g^\lambda$ in a $\Theta$-frame $\{\mbox{\boldmath $Z$}_I\}$, and let 
$$
g^{(k)}_{IJ}:=\frac{1}{k!}\pa^k_\r g^\lambda_{IJ}\big|_M
$$
be the Taylor coefficients, which are analytic functions on $M$. We consider an ACH metric of the form
$$
\wt g^\lambda_{IJ}=\sum_{k=0}^{8}\r^k g^{(k)}_{IJ}+\r^9 \wt\varphi_{IJ},
$$
which automatically satisfies $E_{IJ}=O(\r^9)$. Then, we consider the equation
\begin{equation}\label{E-nine}
-8\r^{-9}(E_{00}, E_{1\ol1}, E_{01}, E_{11})=0
\end{equation}
for $v=(\wt\varphi_{00}, \wt\varphi_{1\ol1}, \wt\varphi_{01}, \wt\varphi_{11})$.
We shall show that this equation is written in the form \eqref{cauchy} for $m=2$ and satisfies the assumption of Theorem \ref{solution}; then we can conclude that $g^\lambda_{IJ}$ converges
since it gives the Taylor expansion of the solution $v$.

We see that in Lemma \ref{Einstein-tensor} the negligible term which we ignored in the computation of 
$E_{IJ}$ is an analytic function in 
$$
x, \  \r, \ \r(\r\pa_\r)^l\pa^\a_x(\r^9\wt\varphi) \ {\rm for}\ l+|\a|\le2,\ l<2.
$$
Thus, it can be written in the form
$$
f^{(1)}_{IJ}(x, \r)+\r^9 f^{(2)}_{IJ}(x, \r, \{(\r\pa_\r)^l\pa^\a_x(\r\wt\varphi)\}_{l+|\a|\le2,\ l<2})
$$
with analytic functions $f^{(1)}_{IJ}, f^{(2)}_{IJ}$. Then, by Lemma \ref{Einstein-tensor}, we have 
\begin{align*}
-8E_{00}&=I_1(\r\pa_\r)\varphi_{00}+I_2(\r\pa_\r)\varphi_{1\ol1} \\
&\quad + 16\r^4|A|^2+f^{(1)}_{00}(x, \r)
+\r^9 f^{(2)}_{00}(x, \r, \{(\r\pa_\r)^l\pa^\a_x(\r\wt\varphi)\}_{l+|\a|\le2,\ l<2}),
\end{align*}
where 
$$
\varphi_{IJ}=\sum_{k=1}^{8}\r^k g^{(k)}_{IJ}+\r^9 \wt\varphi_{IJ}
$$
and 
$$
I_1(t)=t^2-6t-4, \quad I_2(t)=-4(t-2).
$$
Since $E_{00}=O(\r^9)$, we have
$$
I_1(\r\pa_\r)\Bigl(\sum_{k=1}^{8}\r^k g^{(k)}_{00}\Bigr)
+I_2(\r\pa_\r)\Bigl(\sum_{k=1}^{8}\r^k g^{(k)}_{1\ol1}\Bigr)+16\r^4|A|^2+f^{(1)}_{00}(x, \r)=\r^9 f^{(0)}_{00}(x, \r)
$$
with some analytic function $f^{(0)}_{00}$. Therefore, the equation $-8\r^{-9}E_{00}=0$ is written as
$$
I_1(\r\pa_\r+9)\wt\varphi_{00}+I_2(\r\pa_\r+9)\wt\varphi_{1\ol1}
+F_{00}(x, \r, \{(\r\pa_\r)^l\pa^\a_x(\r\wt\varphi)\}_{l+|\a|\le2,\ l<2})=0
$$
with an analytic function $F_{00}$.

Similarly, the equations $-8\r^{-9}E_{IJ}=0$ for $(I, J)=(1, \ol1), (0, 1), (1,1)$ are respectively written as
\begin{align*}
I_3(\r\pa_\r+9)\wt\varphi_{00}+I_4(\r\pa_\r+9)\wt\varphi_{1\ol1}
+F_{1\ol1}(x, \r, \{(\r\pa_\r)^l\pa^\a_x(\r\wt\varphi)\}_{l+|\a|\le2,\ l<2})&=0, \\
I_5(\r\pa_\r+9)\wt\varphi_{01}
+F_{01}(x, \r, \{(\r\pa_\r)^l\pa^\a_x(\r\wt\varphi)\}_{l+|\a|\le2,\ l<2})&=0, \\
I_6(\r\pa_\r+9)\wt\varphi_{11}
+F_{11}(x, \r, \{(\r\pa_\r)^l\pa^\a_x(\r\wt\varphi)\}_{l+|\a|\le2,\ l<2})&=0,
\end{align*}
where $F_{1\ol1}, F_{01}, F_{11}$ are analytic functions and 
$$
I_3(t)=-t+4, \quad I_4(t)=t^2-6t-8, \quad I_5(t)=(t+1)(t-5), \quad I_6(t)=t(t-4).
$$
Hence the equation \eqref{E-nine} is of the form \eqref{cauchy}, and  we have 
\begin{align*}
\det\mathcal{P}(k)&=\det
\begin{pmatrix}
I_1(k+9) & I_2(k+9) & {} & {} \\
I_3(k+9) & I_4(k+9) & {} & {} \\
{} & {} & I_5(k+9) & {} \\
{} & {} & {} & I_6(k+9) 
\end{pmatrix} \\
&=(k+1)(k+3)(k+4)(k+5)(k+9)^2(k+10)(k+11) \\
&\neq 0
\end{align*}
for any $k\in\mathbb{N}$. Thus, by Theorem \ref{solution} the equation \eqref{E-nine} has a unique 
holomorphic solution and we complete the proof of Theorem \ref{convergence}.

\end{document}